\documentclass[12pt,a4paper]{article}
\usepackage{amsmath,amssymb,amsthm,amsbsy,latexsym}
\usepackage{tikz}
\usetikzlibrary{decorations}
\usetikzlibrary{decorations.pathmorphing}
\usepackage[all]{xy}

\newcommand{\N}{\mathbb{N}}

\newcommand{\Z}{\mathbb{Z}}

\newcommand{\FF}{\mathbb{F}}

\newcommand{\Cc}{\mathcal{C}}

\def\co{\colon\thinspace}

\DeclareMathOperator{\colim}{colim}

\DeclareMathOperator{\Map}{Map}
\DeclareMathOperator{\id}{id}

\newcommand{\pTop}{\mathbf{Top}_\bullet}
\newcommand{\Spt}{\mathbf{Spt}}
\newcommand{\ep}{\mathbf{D}}

\newcommand{\Hp}{\mathrm{H}\mathbb{Z}/p}
\newcommand{\dlc}{\mathcal{Q}}
\newcommand{\brc}{\mathcal{L}}
\newcommand{\stp}{\mathrm{P}}
\newcommand{\units}{\times}
\newcommand{\even}{\, \mathrm{even}\, }
\newcommand{\odd}{\, \mathrm{odd}\, }
\newcommand{\iso}{\cong}
\newcommand{\shift}[1]{\mathrm{sh}^{#1}}

\theoremstyle{plain}
\newtheorem{Theorem}{Theorem} [section]
\newtheorem{Lemma}[Theorem]{Lemma}
\newtheorem{Corollary}[Theorem]{Corollary}
\newtheorem{Proposition}[Theorem]{Proposition}
\newtheorem{Conjecture}[Theorem]{Conjecture}
\newtheorem{nTheorem}[Theorem]{Theorem}
\newtheorem{Definition}[Theorem]{Definition}

\theoremstyle{definition}

\newtheorem{Remark}[Theorem]{Remark}

\makeatletter
\def\qed{%
  \ifmmode 
  \tag*{\qedsymbol}
  \else \leavevmode\unskip\penalty9999 \hbox{}\nobreak\hfill
  \quad\hbox{\qedsymbol}\pagebreak[1]\vskip 2\topsep\fi}

\numberwithin{equation}{section}

\title{The Arone-Goodwillie spectral sequence for $\Sigma^{\infty}\Omega^n$ and topological realization at odd primes}
\author{S.~B\"uscher\footnote{Institut f\"ur Mathematik, Universit\"at Osnabr\"uck, Osnabr\"uck, Germany} \ F.~Hebestreit\footnote{Mathematisches Institut, Universit\"at M\"unster, M\"unster, Germany} \ O.~R\"ondigs\footnotemark[1] \ M.~Stelzer\footnotemark[1]}
\date{November 1, 2011}

\begin{document}

\maketitle

\section{Introduction}
\label{sec:introduction}

A basic task of algebraic topology is to construct meaningful
algebraic invariants of topological spaces. Singular cohomology
can be viewed
as such an invariant, but in many different ways. In order to specify
the algebraic natures of singular cohomology appearing in this article, 
fix a finite prime field $\FF_p$ as coefficients.
A rather elementary 
view is to consider $H^\ast(X)$, the direct sum of all cohomology
groups of a topological space $X$, as a graded $\FF_p$-vector space.
Via cup product, it in addition becomes an $\FF_p$-algebra. Furthermore, $H^\ast(X)$ is a
module over the Steenrod algebra
$\mathcal{A}_p$, the algebra of all stable operations in singular cohomology
with $\FF_p$-coefficients. 
The realization question under consideration
is whether a given algebraic
structure, for example an $\mathcal{A}_p$-module $M$ or
an $\mathcal{A}_p$-algebra $A$, is isomorphic
to $H^\ast(X)$ for some topological space. There are
some obvious restrictions (the module should be unstable),
but also more delicate ones. Nick Kuhn
formulated the following realization conjecture in \cite{Kuhn:realizing}.

\begin{Conjecture}\label{conj:kuhn}
  For any topological space $X$ such that $H^\ast(X)$ is a
  finitely generated $\mathcal{A}_p$-module, 
  $H^\ast(X)$ is already a finite-dimensional $\FF_p$-vector space. 
\end{Conjecture}

Kuhn proved Conjecture~\ref{conj:kuhn} under an
additional assumption on the behaviour of the
Bockstein. 
He employed Lannes' $T$-functor to construct from a hypothetical counterexample
to Conjecture~\ref{conj:kuhn} an even more hypothetical 
topological space $Y$
whose cohomology $H^\ast(Y)$ has a specific finite composition series.
Further ingredients in his proof that such a space $Y$ 
cannot exist are the mod $p$ solution  
of the Hopf invariant one problem, and 
results of Ravenel and Mahowald on the 
mod $p$ Kervaire invariant problem. 

Lionel Schwartz proved the general case 
of Conjecture~\ref{conj:kuhn} for $p=2$ 
(without assumptions on the Bockstein)
in \cite{Schwartz:conjecture}.
Continuing with Kuhn's dubious space $Y$, 
his key idea was to analyze the
induced unstable $\mathcal{A}_p$-algebra structure
of $H^{\ast}(\Omega^n Y)$ for an appropriate $n$ 
via an iterated use of the Eilenberg-Moore spectral sequence.
The article \cite{Schwartz:conjecture} contains a brief discussion
of the odd primary case as well; an erratum \cite{Schwartz:Errata}
pointing to a complete proof in \cite{Gaudens-Schwartz} appeared recently.

As explained by Kuhn in \cite{Kuhn:nonrealization},
the iterated Eilenberg-Moore spectral sequences may be replaced with
a single spectral sequence, defined by the Goodwillie tower of the functor
$X\mapsto \Sigma^{\infty}\Omega^n X$. This approach
lead him to a streamlined proof of Conjecture~\ref{conj:kuhn} 
at the prime 2 via nonrealization theorems 
for certain finite unstable modules 
over the Steenrod algebra $\mathcal{A}_2$.

The present paper is rooted in our attempt to 
generalize Kuhn's approach in \cite{Kuhn:nonrealization} to odd primes.
Our study of the Arone-Goodwillie spectral sequence at an odd prime 
is parallel to Smith's classical investigation of the 
Eilenberg-Moore spectral sequence in \cite{Smith:lectures}.
In order to state our main theorem, let
$\Phi(k,k+2)$ be the quotient of the sub-$\mathcal{A}_p$-module of
$H^\ast\bigl(K(\FF_p,1)\bigr)$ generated by 
an element $t\neq 0$ with $\lvert t\rvert =2$, having $\FF_p$-basis
$\{t^{p^k},t^{p^{k+1}},t^{p^{k+2}}\}$. The module $\Phi(k,k+2)$
is closely related to the free unstable module $F(1)$ on a
one-dimensional class described in \cite{Kuhn:realizing}.
The notion of
a desuspension class of even origin is given in
Definition~\ref{def:des-class}.

\begin{nTheorem}\label{thm:main-intro}
  Let $M$ be an unstable $\mathcal{A}_p$-module of finite type, 
  concentrated in degrees $[\ell,m]$. Suppose $M$
  contains a  
  desuspension class of even origin.
  If $X$ is a topological space with
  \[ \widetilde H^\ast(X) \cong M \otimes \Phi(k,k+2) \]
  as $\mathcal{A}_p$-modules, then $2p^k \leq (p^2-1)m + p(m-\ell)$.
\end{nTheorem}

Theorem~\ref{thm:main-intro}
generalizes Kuhn's nonrealization theorems mentioned above
only in roughly half of all the cases.
The reason for this may be traced back to the fact that, 
at odd primes, the parity of
certain elements determines the
type of operation detecting
the desuspension behaviour of a given unstable 
$\mathcal{A}_p$-module $M$. 
Only in the case the Bockstein is not involved, a direct relation to
an unstable algebra structure on $M$ can be made.
The arguments leading to the nonrealization theorems use
this relation in the unstable algebra $H^{\ast}(\Omega^n X)$
in an essential way. 
Theorem~\ref{thm:main-intro} implies a new proof of 
Kuhn's conjecture~\ref{conj:kuhn} in the case that the 
Bockstein
acts trivially in high degrees.

Gaudens and Schwartz
gave a proof
of a generalized form of Kuhn's  conjecture~\ref{conj:kuhn} 
for all primes~\cite{Gaudens-Schwartz}.  
However, contrary to Theorem \ref{thm:main-intro}, 
their work gives no
information about topologically realizing  
finite modules by the very nature of their argument.

The paper is organized as follows. Section~\ref{sec:arone-goodw-spectr} 
recalls some facts about the spectral sequence from work of Goodwillie, 
Arone, Ahearn and Kuhn. Besides the right $\mathcal{A}_p$-action, 
the homology spectral sequence comes with an action
of Dyer-Lashof and  Browder operations. 
The interplay between these structures is essential to our proofs. 
In order to employ algebra structures, we switch to cohomology.
This forces us to define and study operations dual to the
corresponding homology operations, which is presented 
in section~\ref{sec:operations}.
The first nontrivial differentials are 
investigated and related to the operational structure in 
section~\ref{sec:some-differentials}.
We prove our nonrealization theorems in section~\ref{sec:nonr-results}. 
They improve on the existing results in the literature. 

Finally there are two appendices. In appendix~\ref{sec:homol-free-algebr} 
the cohomology of a free algebra over a suitable $E_n$ operad in spectra is computed by reduction
to Cohen's unstable calculation~\cite{C.L.M.}. 
This result is needed in section~\ref{sec:nonr-results}. 
Appendix~\ref{sec:nishida-relations}
mainly consists of a proof for  
cohomological Nishida relations on the dual operations defined in
section~\ref{sec:operations}, again needed in section \ref{sec:nonr-results}.
   
\subsection{Notation and conventions} 

All (co)homology groups have coefficients in the 
prime field $\FF_p$, where $p$ is an odd prime.
The corresponding Eilenberg-MacLane spectrum
is denoted $\Hp$.
Unless stated otherwise, topological spaces have the homotopy type
of cell complexes.
From Section~\ref{sec:operations} on 
all spectra or topological spaces $X$ are bounded 
below with $H^\ast(X)$ of finite type, i.e., finite-dimensional
in each degree. The $n$-fold shift of a graded module $M$ is denoted
$\shift{n}M$ and satisfies $\bigl(\shift{n}M\bigr)^{k} = M^{k-n}$.
Tensor products are over $\FF_p$, unless stated otherwise.

\subsection{Acknowledgements}

We are grateful to Nick Kuhn for several discussions,
and for providing us with the Thom isomorphism argument 
which is central to the results in 
Appendix~\ref{sec:homol-free-algebr}.
We also thank Gerald Gaudens and Lionel Schwartz for 
sending us an early draft of \cite{Gaudens-Schwartz}, as
well as for stimulating discussions we had with them in 
Osnabr\"uck and Paris.

\section{The Arone-Goodwillie spectral sequence}
\label{sec:arone-goodw-spectr}

Goodwillie's calculus of homotopy functors, as developed in \cite{Go}, \cite{Go2} and \cite{Go3}, 
provides an interpolation between 
unstable and stable phenomena in homotopy theory. The layers of this
interpolation carry a surprisingly rich structure. 
The purpose of this section is to describe the interpolation in the
examples relevant to us, following work of Goodwillie, Johnson and
Arone.

Let $\pTop$ be a convenient category of pointed topological spaces (for example, compactly
generated spaces), and
let $\Spt$ be a convenient category of $S^1$-spectra in $\pTop$ (for example, orthogonal spectra).
Suppose $F\co \Cc \to \Spt$ is a functor preserving weak equivalences, 
where $\Cc$ is either the category $\pTop$ or $\Spt$. 
Goodwillie's calculus of functors provides a sequence of functors $P_nF\co \Cc\to \Spt$ 
and a natural tower of fibrations for any $X\in \Cc$:
\[\begin{xy}
   \xymatrix{
                                                                 & & \ar[d] \\
                                                                 & & P_{k+1}F(X)\ar[d]^{p_{k+1}F} \\
                                                                 & & P_{k}F(X) \ar[d]^{p_{k}F} \\
                                                                 & & \ar@{.}[d] \\
                                                                 & & \ar[d]^{p_1F} \\
       F(X) \ar[rr]^{f_0} \ar[rruuu]_{f_{k}} \ar[rruuuu]^{f_{k+1}} & & P_0F(X) \\
   }
\end{xy}\]
The transformation  $F\to P_kF$ is universal among $k$-excisive approximations of $F$.
The fiber $D_kF$ of $P_kF \xrightarrow{p_kF} P_{k-1}F$ is called the $k$-th differential of $F$
and may be thought of as the degree-$k$ homogeneous part of the polynomial functor $P_kF$.
Applying any cohomology functor $h^\ast$ to the Goodwillie tower produces an exact couple and
hence a spectral sequence $(E,d)$ with:
\begin{equation}\label{eq:goodwilliess}
  E_1^{-j,k} = \begin{cases}
    h^{k-j}(D_jF(X)) &  j \geq 0 \\
    0               &  j < 0
  \end{cases}
\end{equation}
This Goodwillie spectral sequence will be exploited for the functors
\[ \Sigma^\infty \Omega^n\co \pTop \to \Spt \quad \mathrm{and} \quad
   \Sigma^\infty\Omega^\infty \co \Spt \to \Spt\]
composed with $h^\ast = H^\ast(-;\FF_p)$.
Arone gave an explicit model of the Goodwillie tower for 
the functor 
\[ \Sigma^\infty \Map_\bullet(K,-)\co \pTop \rightarrow \Spt\]
with $K$ a finite (pointed) cell complex in \cite{Arone:snaith}. This model was used in \cite{A.K.} to describe
the $k$-th differential 
\[ D_{n,k} \co = D_k\Sigma^\infty \Omega^n \]
of the Goodwillie tower of 
$\Sigma^\infty \Omega^n = \Sigma^\infty\Map_\bullet(S^n,-)$
via the operad of little cubes introduced by 
Boardman and Vogt in \cite{Boardman-Vogt:invariant}. In order to present this description,
let $\Cc(n,k)$ be the space of $k$ disjoint little $n$-cubes in a bigger
$n$-cube (cf. \cite{May:geometry}), with the obvious action of the symmetric group $\Sigma_k$. 

\begin{Theorem}\label{fibDn} Let $1\leq n\in \mathbb{N}$.
\begin{enumerate}
   \item For any pointed topological space $X$, there is a natural weak equivalence of spectra
     \[ D_{n,k} (X) \simeq {\Cc(n,k)_+} \wedge_{h\Sigma_k} (\Sigma^{-n} \Sigma^\infty X)^{\wedge k}\]
   \item For any spectrum $X$, there is a natural weak equivalence of spectra
     \[ D_{\infty,k} (X) \simeq {\Cc(\infty,k)_+} \wedge_{h\Sigma_k} X^{\wedge k}\]
\end{enumerate}
\end{Theorem}

\begin{proof}
  The first weak equivalence is stated as~\cite[Eqn.~(1,2)]{A.K.} and is obtained as
  a consequence of~\cite[Thm.~7.1]{A.K.}. The second weak equivalence is \cite[Cor.~1.3]{A.K.}.
\end{proof}

The special cases $D_{1,k}(X) \simeq X^{\wedge k}\!$, 
$D_{\infty,k}(X) \simeq X^{\wedge k}_{h\Sigma_k}$, 
$D_{n,1}(X) \simeq \Sigma^{-n}\Sigma^\infty X$ follow
from various properties of the spaces $\Cc(n,k)$. Furthermore, by convention
$D_{n,0}(X) = \mathbb S$
is the sphere spectrum.
The case $n=\infty$ is used for universal examples in upcoming proofs.
In order to simplify the exposition, it will sometimes
be included without explicit notational changes.

Note that convergence of the Goodwillie tower requires further assumptions on the variable $X$.
For example, in the first case of Theorem~\ref{fibDn}, convergence holds by \cite{Arone:snaith} on
$n$-connected pointed topological spaces, whereas in the second case, 
convergence holds on $0$-connected spectra~\cite[Cor.~1.3]{A.K.}.
The Arone-Goodwillie spectral sequence arising from $\Sigma^\infty \Omega^n$ 
in singular cohomology with $\FF_p$ coefficients will simply be referred to as 
'the spectral sequence for $n$', infinity included.

\begin{Theorem}\label{sseq}
  The spectral sequence for $n$ converges to $H^\ast(\Omega^n X)$ if $X$ is an $n$-connected pointed 
  cell complex ($n < \infty$) respectively a connected spectrum ($n=\infty$).
  Furthermore, it is a spectral sequence of $\mathcal{A}_p$-modules, where 
  $\mathcal{A}_p$ acts columnwise on each page and
  the $\mathcal{A}_p$-module structure on the $E_1$-page 
  is the obvious one.
  The spectral sequence converges to the 
  usual $\mathcal{A}_p$-module structure on $H^\ast(\Omega^n X)$.
\end{Theorem}

These spectral sequences are tied together as 
$n$ varies. More precisely, there are natural maps
\begin{equation}\label{eq:evaluation}
\epsilon_{r,n}\co \Sigma^{r} \Sigma^\infty \Omega^{n+r} X \rightarrow \Sigma^\infty \Omega^n X
\end{equation}
induced by the evaluation $\Sigma \Omega X \rightarrow X$ for any pointed topological space $X$,
as well as natural equivalences
\begin{equation}\label{eq:suspension}
\varphi_n\co \Sigma^\infty \Omega^n \Omega^\infty \Sigma^n X \rightarrow \Sigma^\infty \Omega^\infty X\end{equation}
for any spectrum $X$. These maps induce maps on the Goodwillie towers, hence, 
of the associated spectral sequences. 
In particular, they induce maps on the differentials, which will be of use later on.

\section{Operations}
\label{sec:operations}

In this section,  following Kuhn's treatment for the case $p=2$, we 
define and study operations in the cohomology of
$\mathcal{C}_n$-extended powers $D_{n,j}X$ in the category of spectra 
which are
closely related to the Browder operations,
Dyer-Lashof operations and Pontryagin products in homology. 
The standard references for these homology operations
are \cite{C.L.M.} and \cite{B.M.M.S.}, with some corrections to be found in 
\cite{Wellington}. We refer to these sources for the
properties of and relations between the homology operations.
To be more precise, the reindexed operations $Q^r$ are studied in loc.cit.,
but the operations which appear as $Q_r$ or $P_r$ in the literature are more convenient
in our setup. We employ the Dyer-Lashof notation $Q_{(n-1)(p-1)},Q_{(n-1)(p-1)-1}$ 
for the two top operations denoted by $\xi$ and $\zeta$ in  \cite{C.L.M.}.
Of course one should keep in mind that $\xi$ is the single nonlinear operation
among these. Furthermore, the 
Bockstein of $\xi$ is $\zeta$ only up to a 
term consisting of a sum of $p$-fold Browder operations if $n>1$. 

\subsection{Definitions}
\label{sec:definitions}

We have to fix notation for the structure maps which show up 
in the definitions of the operations.

\begin{Definition}\label{def:operation-maps}
  In the following, $D_{n,k}(X)$ is equipped with the $\Cc_n$-operad action via the 
  weak equivalence given in Theorem~\ref{fibDn}.
  \begin{enumerate}
  \item\label{item:epsilon} Let \[\epsilon\co
    \Sigma^{r}D_{n+r,j}X \rightarrow D_{n,j}\Sigma^{r}X\] denote the map induced by the
    evaluation map $\Sigma^{r}\Sigma^{\infty}\Omega^{n+r}X \rightarrow \Sigma^{\infty}\Omega^{n}X$
    from~(\ref{eq:evaluation}).
  \item\label{item:mu} Let \[\mu\co D_{n,i}X \wedge D_{n,j}X \rightarrow D_{n,i+j}X\]
    be the map induced by the operad action $\theta$ of some fixed element $c_2\in  \mathcal{C}_{n,2}$ 
    \[\theta (c_2 ,-,-)\co\mathcal{C}_{n,i}\times\mathcal{C}_{n,j} \longrightarrow\mathcal{C}_{n,i+j}\]
    (see \cite[Chapter VII]{L.M.S.} for details). 
  \item\label{item:t} Let
    \[\tau\co D_{n,i+j}X \rightarrow D_{n,i}X \wedge D_{n,j}X\]
    denote the composition of the maps
    \[(\mathcal{C}(n,i+j)_{+}\wedge X^{\wedge i+j})_{h\Sigma_{i+j}} 
    \rightarrow (\mathcal{C}(n,i+j)_{+}\wedge X^{\wedge i+j})_{h\Sigma_{i}\times \Sigma_{j}}\]
    and
    \[ (\mathcal{C}(n,i+j)_{+}\wedge X^{\wedge i+j})_{h\Sigma_{i}\times \Sigma_{j}} 
    \rightarrow (\mathcal{C}(n,i)_{+} \wedge \mathcal{C}(n,j)_{+} \wedge 
    X^{\wedge i+j})_{h\Sigma_{i}\times \Sigma_{j}}\]
    of homotopy orbits, where the first map is the transfer associated to 
    $\Sigma_{i}\times \Sigma_{j} \subset \Sigma_{i+j}$ and the second map is 
    induced by the $\Sigma_{i} \times \Sigma_{j}$-equivariant inclusion of spaces 
    $\mathcal{C}(n,i+j)\subset \mathcal{C}(n,i)\times \mathcal{C}(n,j)$.
  \item\label{item:omega} 
    Let \[\omega\co D_{n,pj}X \rightarrow D_{\infty,p}D_{n,j}X\] denote the composition of maps
    \[ (\mathcal{C}(n,pj)_{+}\wedge X^{\wedge pj})_{h\Sigma_{pj}} \rightarrow (\mathcal{C}(n,pj)_{+}\wedge X^{\wedge pj})_{h\Sigma_{p} \wr \Sigma_{j}}\]
    and
    \[(\mathcal{C}(n,pj)_{+}\wedge X^{\wedge pj})_{h(\Sigma_{p} \wr \Sigma_{j})} \rightarrow (\mathcal{C}(n,j)_{+}^{ p}  \wedge X^{\wedge pj})_{h(\Sigma_{p} \wr \Sigma_{j})}\]
    of homotopy orbits, where the first map is the transfer associated to 
    $\Sigma_{p} \wr \Sigma_{j} \subset \Sigma_{pj}$ and the second map is induced by the 
    $\Sigma_{p} \wr \Sigma_{j}$-equivariant inclusion of spaces 
    $\mathcal{C}(n,pj)\subset \mathcal{C}(n,j)^{p}$.
  \end{enumerate}
\end{Definition}

In the following,  the notion of naturality has to be understood in the context of $E_n$-spectra, 
or more generally of $H_n$-spectra. Note that the corresponding definition in the case $p=2$
is less involved, as then $0=p-2$.

\begin{Definition}\label{defq}
Let $\infty\geq n \geq 1$ and $n \geq r\geq 0$ be given. 
Let $\iota \in H_{d}(\Sigma^{d}\Hp)$ be the fundamental class given by the 
inclusion of the bottom cell.
The natural operation
\[\dlc^{r}\co H^{d}(D_{n,j}X) \rightarrow H^{pd+s}(D_{n,pj}X)\]
is defined as follows.
\begin{description}
\item[$\dlc^0$:]
  For $d$ even and $x\in H^{d}(X)$, let $\dlc^{0} (x) \in H^{pd}(D_{\infty,p}X)$ be the composition
  \[D_{\infty,p}X \xrightarrow{D_{\infty,p}(x)} 
  D_{\infty,p}\Sigma^{d}\Hp \xrightarrow{\hat{\iota}_{0} } \Sigma^{pd}\Hp \]
  where $\hat{\iota}_{0}$ represents the dual of $Q_{0}(\iota) \in H_{pd}(D_{\infty,p}\Sigma^{d}\Hp)$, i.e. the Dyer-Lashof operation evaluated at the 
  fundamental class.
  For $d$ even and $x \in H^{d}(D_{n,j} X)$, let $\dlc^{0}(x) \in H^{pd}(D_{n,pj}X)$ 
  be the composition
  \[D_{n,pj}X \xrightarrow{\omega}D_{\infty,p}D_{n,j}X \xrightarrow{\dlc^{0} (x)} \Sigma^{pd}\Hp\]
  with $\omega$ introduced in Definition~\ref{def:operation-maps}, part~\ref{item:omega}.
  For $x\in H^d(D_{n,j}X)$ with $d$ odd, $\dlc^0(x)$ is defined to be zero.
\item[$\dlc^{p-2}$:]
  For $x\in H^{d}(X)$ with $d$ odd, let 
  $\dlc^{p-2} (x) \in H^{pd+(p-2)}(D_{\infty,p}X)$ be the composition
  \[ D_{\infty,p}X \xrightarrow{D_{\infty,p}(x)} D_{\infty,p}\Sigma^{d}\Hp 
  \xrightarrow{\hat{\iota}_{p-2}} \Sigma^{pd+(p-2)}\Hp\]
  where $\hat{\iota}_{p-2}$ represents the dual of the evaluation of the
  Dyer-Lashof operation 
  $Q_{p-2}(\iota) \in H_{pd+(p-2)}(D_{\infty,p}\Sigma^{d}\Hp)$
  at the fundamental class.
  For $x \in H^{d}(D_{n,j} X)$ with $d$ odd, let $\dlc^{p-2}(x) \in H^{pd+(p-2)}(D_{n,pj}X)$ 
  be the composition
  \[ D_{n,pj}X \xrightarrow{\omega}D_{\infty,p}D_{n,j}X \xrightarrow{\dlc^{p-2} (x)} 
  \Sigma^{pd+(p-2)}\Hp\]
  with $\omega$ introduced in Definition~\ref{def:operation-maps}, part~\ref{item:omega}.
  For $x\in H^d(D_{n,j}X)$ with $d$ even, $\dlc^0(x)$ is defined to be zero.
\item[$\dlc^{s(p-1)}$:] Let $0<s \leq n-1$ and $x\in H^d(D_{n,j}X)$, with $d+s$ even.
  Define $\dlc^{s(p-1)}(x)\in H^{pd+s(p-1)}(D_{n,pj}X)$ to be the composition
  \[D_{n,pj}X \xrightarrow{\epsilon} \Sigma^{-s} D_{n-s,pj}\Sigma^{s} X 
  \xrightarrow{\Sigma^{-s}\dlc^{0}(\Sigma^{s}x)} \Sigma^{pd+s(p-1)}\Hp\]
  with $\epsilon$ introduced in Definition~\ref{def:operation-maps}, part~\ref{item:epsilon}.
  For $x\in H^d(D_{n,j}X)$ with $d+s$ odd, $\dlc^0(x)$ is defined to be zero.
\item[$\dlc^{s(p-1)-1}$:] Let $0<s \leq n-1$ and $x\in H^d(D_{n,j}X)$, with $d+s$ even.
  Define $\dlc^{s(p-1)-1}(x)\in H^{pd+r(p-1)-1}(D_{n,pj}X)$ to be the composition
  \[ D_{n,pj}X \xrightarrow{\epsilon} \Sigma^{1-s} D_{n+1-s,pj}\Sigma^{(s-1)} X 
  \xrightarrow{\Sigma^{1-s }\dlc^{p-2}(\Sigma^{(s-1)}x)} \Sigma^{pd+s(p-1)-1}\Hp\]
  with $\epsilon$ introduced in Definition~\ref{def:operation-maps}, part~\ref{item:epsilon}.
  For $x\in H^d(D_{n,j}X)$ with $d+s$ odd, $\dlc^0(x)$ is defined to be zero.
\item[$\dlc^r$:]
  All other $\dlc^{r}$ are defined to be zero.
\end{description}
\end{Definition}

Thus the only (possibly) nonzero dual Dyer-Lashof operations are: 
\[\dlc^0,\dlc^{p-2},\dlc^{p-1},
\dlc^{2p-3},\dlc^{2p-2},\dotsc,\dlc^{(n-1)(p-1)}\] 
The dual Browder operations are
obtained as follows.

\begin{Definition}\label{defl}
  Let  $d=\sum_{i=1}^j d_{i}$. The dual Browder operation is the homomorphism
  \[ \brc^{n-1}\co H^{d_{1}}(X) \otimes \cdots \otimes H^{d_{j}}(X) 
  \rightarrow H^{d+(j-1)(n-1)}(D_{n,j}X)\]
  induced by the map 
  \[ \epsilon^{n-1} \co D_{n,j} X \rightarrow \Sigma^{1-n}D_{1,j}\Sigma^{n-1}X=
  \Sigma^{(j-1)(n-1)}X^{\wedge j} \]
  with $\epsilon$ introduced in Definition~\ref{def:operation-maps}, part~\ref{item:epsilon}.
\end{Definition}

Note that $\brc^0$ coincides with the K\"unneth isomorphism.
Our final definition in this section describes a ``product'' resp. a ``coproduct'' structure
on $\bigoplus_k H^\ast(D_{n,k}X)$.

\begin{Definition}\label{def:pairing-copairing}
  The map $t$ introduced in Definition~\ref{def:operation-maps}, part~\ref{item:t}
  induces a pairing:
  \[ \star\co H^{*}(D_{n,i}X) \otimes H^{*}(D_{n,j}X) \rightarrow H^{*}(D_{n,i+j}X)\]
  The map $\mu$ introduced in Definition~\ref{def:operation-maps}, part~\ref{item:mu}
  induces a copairing:
  \[ \Psi\co H^{*}(D_{n,i+j}X)  \rightarrow H^{*}(D_{n,i}X) \otimes H^{*}(D_{n,j}X)\]
\end{Definition}

\subsection{Properties}
\label{sec:properties}

This section summarizes some properties of the operations and pairings defined in
Section~\ref{sec:definitions}.

\begin{Proposition}\label{epsilonCO}
  The map $\epsilon^\ast\co  H^\ast(D_{n,j}\Sigma X) \rightarrow \shift{1} H^\ast(D_{n+1,j} X)$
  commutes with the dual Dyer-Lashof and dual Browder operations. More precisely,
  the following equalities hold:
  \begin{align*}
    \epsilon^\ast(\dlc^r(\sigma x)) & = \dlc^{r+p-1}(x) \\
    \epsilon^\ast(\brc^{n-1}(\otimes \sigma(x_i)) &= \brc^n(\otimes x_i)
  \end{align*}
\end{Proposition}

\begin{proof}
  This is a direct consequence of Definitions \ref{defq} and \ref{defl}.
\end{proof}

Fix $1\leq n\in \N$.
In order to recall some properties of the Dyer-Lashof operations 
\[ Q_r\co H_\ast(D_{n,j} X) \rightarrow H_\ast(D_{n,pj} X), \]
the Browder operations 
\[L_{n-1}\co H_\ast(D_{n,i} X) \otimes H_\ast(D_{n,j} X) \rightarrow H_*(D_{n,i+j}X)\]
and the Pontryagin product  
\[\ast\co H_\ast(D_{n,i} X)\otimes H_\ast(D_{n,j} X)\to H_\ast(D_{n,i+j} X)\]
from the literature, let  $q\in \mathbb{N}$ and 
$r=t(p-1)+k$ with $0\leq k < p-1$. Set
\begin{subequations}
  \begin{align}
    \gamma(q,r)&=\prod_{i=0}^{t-1}(-1)^{\tfrac{(p-1)(q+i)}{2}}\Bigl(\bigl(\tfrac{p-1}{2}\bigr)!\Bigr)^{t} \label{eq:gamma}\\
    \nu(q)&= (-1)^{\tfrac{(p-1)q}{2}}\bigl(\tfrac{p-1}{2}\bigr)!\label{eq:nu}\\
    \lambda(q) &= \prod_{i=0}^{n-2}\nu(q+i)\label{eq:lambda}.
  \end{align}
\end{subequations}
The (co)homological degree of a homogeneous (co)homology class $y$ 
is denoted $\lvert y \rvert$.

\begin{Proposition}\label{epsilonHO}
  The map 
  $\epsilon_\ast\co \shift{1} H_\ast(D_{n+1,j} X) \rightarrow H_\ast(D_{n,j} \Sigma X)$
  satisfies:
  \begin{enumerate}
  \item\label{item:eps-q} ${\epsilon}_\ast\bigl(\sigma Q_{r}(y)\bigr) = 
    \nu(\lvert y\rvert ) Q_{r-(p-1)}\bigl(\sigma(y)\bigr)$. 
    Here $Q_i$ is understood to be trivial for negative $i$.  
  \item\label{item:eps-l} For all $x\in H_\ast(D_{n,i}X)$ and all $y\in H_\ast(D_{n,i}X)$ 
    one has ${\epsilon}_\ast\bigl(\sigma L_{n}(x , y)\bigr) = 
    L_{n-1}(\sigma x , \sigma y)$.
  \item\label{item:eps-prod} All Pontryagin products of 
    elements in positive degrees  are in the kernel of ${\epsilon}_\ast$.  
  \end{enumerate}
\end{Proposition}

\begin{proof} 
  The sources quoted above consider the operations in the homology of loop spaces,
  not $\Cc_n$-spectra.
  As explained in Appendix~\ref{sec:homol-free-algebr},
  one can reduce to this situation  
  via a suitable Thom isomorphism.
  For parts~\ref{item:eps-q} and~\ref{item:eps-l} 
  see \cite[III.1.4.]{C.L.M.}, which discusses all 
  cases except $Q_{(n-1)(p-1)}=\zeta$. 
  The remaining case is treated via the equality 
  $\zeta (x) = \beta\xi (x) - ad^{p-1}(x)\bigl(\beta (x)\bigr)$,
  where $\beta\in \mathcal{A}_p$ is the Bockstein. 
  The assertion in part~\ref{item:eps-prod} is classical and can be found for example 
  in \cite{Whitehead:elements}.
\end{proof}

\begin{Proposition} \label{tStern}
  The following hold for $\tau_\ast\co H_\ast(D_{n,k} X) \rightarrow H_\ast(X)^{\otimes k}$:
  \begin{subequations}
    \begin{align}
      \tau_\ast(Q_{r}(y)) &= 0 \label{eq:tq}\\
      \tau_\ast(\ast_{i=1}^k y_i ) &= \sum_{\sigma \in \Sigma_k} \otimes_{i=1}^k y_{\sigma(i)}
        \label{eq:tp}\\
      \tau_\ast(L_{n-1}(y,z)) &= 0 \label{eq:tl}
    \end{align}
  \end{subequations}
\end{Proposition}

\begin{proof}
  This follows immediately from the definitions.
\end{proof}

In the following  $L_{n-1}(y_1 ,y_2 ,\ldots ,y_k )$ denotes any $k$-fold iteration of Browder operations
on elements $y_{1},\dotsc,y_n \in H_{\ast}(X)$.

\begin{Proposition}\label{duality} 
  Let $k\in \mathbb{N}$, $x, x_{1},\dotsc,x_k \in H^\ast(X)$, $y, y_{1},\dotsc,y_k \in H_{\ast}(X)$, 
  $w \in H^{\ast}(D_{n,k} X)$, $z \in H_\ast(D_{n,k} X)$,
  and $z_1,\dotsc,z_k \in \bigoplus_j H_\ast(D_{n,j} X)$ with
  $\ast_{j=1}^k z_j\in H_\ast(D_{n,k} X)$. 
  Under the Kronecker pairing the cohomology and the 
  homology operations pair as follows: 
  \begin{enumerate}
  \item\label{item:qq} $ \langle \dlc^r(w), Q_s(z)  \rangle = 
    \begin{cases}
      \gamma(\lvert w\rvert ,r) \langle w, z \rangle   &   r = s = t(p-1)-\epsilon,\,  \epsilon\in\{0,1\}  \\
      0                                 & \mathrm{otherwise}
    \end{cases}$
  \item\label{item:ql} $ \langle \dlc^r(x), L_{n-1}(y_1 ,y_2 ,\ldots ,y_k ) \rangle  =  0 $
  \item\label{item:qp} $\langle \dlc^r(x), *_{i=1}^k y_i \rangle  = \begin{cases}
      \prod_{i=1}^k \langle x, y_i \rangle
      &  r = 0 \  \mathrm{and}\  k=p \\
      0                                                  &  \mathrm{otherwise}
  \end{cases}$
  \item\label{item:lq} $ \langle \brc^{n-1}(\otimes_{i=1}^k x_i), Q_s(y) \rangle  =
    \begin{cases}
      \lambda(\vert y\rvert )\! \prod_{i=1}^p \!\langle x_i, y \rangle  &  k=p,s=(n-1)(p-1) \\
      0                                                  &  \mathrm{otherwise}
    \end{cases}$
  \item\label{item:ll} $  \langle \brc^{n-1}(\otimes_{i=1}^k x_i), L_{n-1}(y_1 ,y_2 ,\ldots ,y_{k} )
    \rangle = \langle \otimes_{i=1}^k x_i, L_{0}(y_1 ,y_2 ,\ldots ,y_{k} )\rangle$
  \item\label{item:lp} $ \langle \brc^{n-1}(\otimes_{i=1}^k x_i), \ast_{j=1}^k z_j \rangle  =  0 $
  \item\label{item:pq} $  \langle \star_{i=1}^p x_i, Q_s(y) \rangle = 0 $
  \item\label{item:pl} $ \langle \star_{i=1}^k x_i , L_{n-1}(y_1 ,y_2 ,\ldots ,y_k) \rangle = 0 $
  \item\label{item:pp} $ \langle \star_{i=1}^k x_i, \ast_{i=1}^k y_i \rangle = 
    \sum_{\sigma \in \Sigma_k} \prod_{i=1}^k \langle x_i, y_{\sigma(i)} \rangle$
  \end{enumerate}
\end{Proposition}

\begin{proof}
  \begin{enumerate}
  \item This  follows from the definitions and part~\ref{item:eps-q} of 
    Proposition~\ref{epsilonHO}, together with the congruence
    $(m!)^2=(-1)^{m+1} \mod p$, which is a consequence of Wilson's theorem.
  \item The morphism of operads  \[\eta_n \co \mathcal{C}_n\to \mathcal{C}_{n+1}\] 
    from \cite[p.~31]{May:geometry} induces an $H_n$-map
    \[s(n) \co D_{n,i}X\to D_{n+1,i}X\]
    by \cite{B.M.M.S.}.
    Any $x \in H^\ast(X)\iso H^{\ast}(D_{n,1}X)$ can be written as 
    $x=s(n)^{\ast}x^\prime$ for $x^\prime \in H^\ast(D_{n+1,1}X) \iso H^\ast(X)$. Then  
    \begin{align*} 
      \langle \dlc^{r}(x), L_{n-1}(y_1 ,y_2 ,\ldots ,y_k) \rangle &= 
      \langle \dlc^{r}(s(n)^{\ast}x^\prime), L_{n-1}(y_1 ,\ldots ,y_k )  \rangle \\
      &= \langle s(n)^{\ast} \dlc^{r}(x^\prime),  L_{n-1}(y_1 ,\ldots ,y_k )  \rangle \\
      &= \langle \dlc^{r}(x^\prime),s(n)_\ast  L_{n-1}(y_1 ,\ldots ,y_k )    \rangle \\
      &= \langle \dlc^{r}(x^\prime),  L_{n-1}(s(n)_\ast y_1 , \dotsc ,s(n)_\ast y_k )   \rangle \\
      &= \langle \dlc^{r}(x^\prime),  0 \rangle =0 
    \end{align*} 
    because the Browder operations $L_{n-1}$ vanish on $\mathcal{C}_{n+1}$-spectra by their very nature.
  \item For $r$ different from $0$ or $p-2$ the assertion follows from the definition
    of $\dlc^r$ and part~\ref{item:eps-prod} of Proposition~\ref{epsilonHO}. 
    Since $\dlc^0$ is essentially the dual of the $p$-fold 
    Pontryagin product operation, the assertion holds for $r=0$.
    For $r=p-2$ it suffices to consider classes of odd degree. Since $\iota^2 =0$ for any 
    generator $\iota \in H_{d}(D_{\infty ,p}\Sigma^d \Hp )$ if $d$ is odd, 
    the group $H_{pd+p-2}(D_{\infty ,p}\Sigma^d \Hp )$ is generated by
    the indecomposable element $Q_{p-2}(\iota)$.
    The assertion for $r=p-2$ follows. 
  \item This follows from part~\ref{item:eps-q} of Proposition~\ref{epsilonHO} and the 
    fact that $Q_0$ is the Frobenius.
  \item This follows from Proposition~\ref{epsilonCO}, part~\ref{item:eps-l} of Proposition~\ref{epsilonHO}
    and the definition of a graded commutator, which describes $L_0$.
  \end{enumerate}
  The  assertions in \ref{item:lp}, \ref{item:pq}, \ref{item:pl} and \ref{item:pp} 
  follow from part~\ref{item:eps-prod} of Proposition~\ref{epsilonHO}, 
  (\ref{eq:tq}), (\ref{eq:tl}) and (\ref{eq:tp}) 
  respectively.
\end{proof}

\begin{Proposition}\label{qcal}
  Let $x,y \in H^\ast(D_{n,j} X)$ and $r\in \mathbb{N}$. Then:
  \[ \dlc^r(x+y) = \begin{cases}
    \dlc^r(x) + \dlc^r(y)  &r>0 \\
    \dlc^0(x) +  \dlc^0(y) + \sum\limits_{k=1}^{p-1} \frac{1}{(p-1)!k!} x^{\star k} \star y^{\star(p-k)} & r=0
  \end{cases}
  \]
  In particular, $\dlc^r(u\cdot x) = u \cdot \dlc^r(x)$ for every $u\in \FF_p$.
\end{Proposition}

\begin{proof}
  The proof is similar to the proof of Proposition~\ref{nish}
  and will be given in Appendix~\ref{sec:nishida-relations}.
\end{proof}

In order to interpret the following proposition for all $n$ including $\infty$, note that
$\brc^{n-1}(x)=x$ for any $x$, and that $\brc^\infty(\otimes_{i=1}^k x_i)=0$
whenever $k>1$. 

\begin{Proposition}\label{prop:generated-alg}
Let  $0\leq j <p^{2}$. The cohomology $H^\ast(D_{n,j} X)$ is generated as an algebra
by elements $\dlc^{r}\brc^{n-1}(\otimes_{i=1}^k x_i)$ and 
$\brc^{n-1}(\otimes_{i=1}^k x_i)$ where $r,k\in \N$ and $x_1,\dotsc,x_k \in H^\ast(X)$. 
The degrees of these generators are:
\begin{align*} 
\lvert \brc^{n-1}(\otimes_{i=1}^k x_i)\rvert &=(k-1)(n-1)+ \sum_{i=1}^k \lvert x_{i}\rvert \\
\lvert \dlc^{r}\brc^{n-1}(\otimes_{i=1}^k x_i)\rvert  &= 
p\Bigl(r+(k-1)(n-1)+\sum_{i=1}^k \lvert x_{i}\rvert \Bigr)
\end{align*} 
\end{Proposition}

\begin{proof}
  This follows from Corollary~\ref{cor:div-power} which describes 
  $H^\ast(D_{n,\ast} X)$ as a free algebra with divided powers, generated by 
  the duals of a generating set of  $H_\ast(D_{n,\ast} X)$. This generating
  set is given by Dyer-Lashof operations applied to iterated Browder products
  of generators of $H_\ast(X)$. The dualities described in proposition \ref{duality} 
  guarantee 
  that all generators in this generating set are detected by the elements given above.
  The statement about the degrees follows from the definition.
\end{proof}

The discussion of the dual operations concludes with Nishida-type relations. 
Due to its length,
the proof is deferred to 
Appendix~\ref{sec:nishida-relations}. 

\begin{Proposition}\label{nish}
  Let $s\in \N$ and $0<r<(p-1)(n-1)$. The equation
  \begin{equation}\label{eq:nish1} 
    \begin{split}
      \stp^s  \dlc^r (x)  & =  \sum_{i=0}^{\bigl\lfloor \frac{s}{p}\bigr\rfloor} a^{r,s,\lvert x\rvert}(i) c^{r,s,\lvert x\rvert}(i)  \dlc^{r+2(s-pi)(p-1)} \stp^i(x) \\
    & +  \delta^r\sum_{i=0}^{\bigl\lfloor \frac{s-1}{p} \bigr\rfloor}b^{r,s,\lvert x\rvert }(i) d^{r,s,\lvert x\rvert}(i)  \dlc^{r+2(s-pi)(p-1)-p}  \stp^i\beta (x) \\
   & + \xi^r \sum_{n \in N} \tfrac{1}{e_n}\bigl(\star_{i=1}^p  \stp^{n_i}(x)\bigr)
    \end{split}
  \end{equation}
  hold, where
  \begin{align*}
    \delta^r  & = \begin{cases} 0 & r \equiv 1 \mod 2 \\ 1 & r \equiv 0 \mod 2
    \end{cases} \\
    \xi^r  & = \begin{cases} 0 & r >0 \\ 1 & r =0
    \end{cases} \\
    a^{r,s,\lvert x\rvert}(i)   &=  \binom{\bigl\lfloor\frac{r}{2}\bigr\rfloor+\frac{(\lvert x \rvert -2i)(p-1)}{2}}{s-pi}\\
    b^{r,s,\lvert x\rvert}(i)  &= (-1)^{\frac{p-1}{2}(\lvert x \rvert +1)+1} \bigl(\tfrac{p-1}{2}\bigr)! \binom{\bigl\lfloor\frac{r}{2}\bigr\rfloor+\frac{(\lvert x \rvert-2i)(p-1)}{2}-1}{s-pi-1} \\
    c^{r,s,\lvert x\rvert}(i)  &= \frac{\gamma(\lvert x \rvert ,r)}{\gamma(\lvert x \rvert +2(p-1)i,r-2(pi-s)(p-1))}\\
    d^{r,s,\lvert x\rvert}(i)  &= \frac{\gamma(\lvert x \rvert ,r)}{\gamma(\lvert x \rvert +1+2(p-1)i,r-p-2(pi-s)(p-1))} \\
    N &=\bigl\{n \in \N^p \co n_i \leq n_{i+1}, \sum n_i = s, \exists j\co n_j < n_{j+1}\bigr\}
  \end{align*}
  and $e_n$ is the residue class in $\FF_p$ of
  the order of the isotropy group of $n\in N$, where
  $\Sigma_p$ acts on $\N^p$ by permuting coordinates.
\end{Proposition}

The point of the following lemma is that -- contrary to a naive guess -- every  summand 
appearing in the right hand side
contains less than $p$ factors $\stp^{p^k}$. This fact is essential to the proof of 
Theorem~\ref{thm:nneq}.

\begin{Lemma}\label{subalg} 
  Let $\mathcal{A}_p(k)\subset \mathcal{A}_p$ be the subalgebra generated by 
  the elements $1, \stp^1, \stp^p, \dotsc, \stp^{p^k}$.
  There is a decomposition
  \[ \stp^{p^k}\stp^{(p-1)p^k} = \sum_{j=1}^r  A_{j1} \stp^{p^k} A_{j2}\stp^{p^k}\dotsm A_{je_j} \]
  with $A_{ji}\in \mathcal{A}_p(k-1)$ and $0\leq e_j\leq p$ for all $1\leq j\leq r$.
\end{Lemma}

\begin{proof}
  To provide the decomposition, note that
  \begin{equation}\label{eq:steenrod-relation}
    \stp^{p^k}\stp^{(p-1)p^k} = \sum_{i=0}^{p^{k-1}} (-1)^{p^k+i}
    \binom{(p-1)\bigl((p-1)p^k-i\bigr)-1}{p^k-pi}\stp^{p^{k+1}-i}\stp^i
  \end{equation}
  by the Adem relations. The binomial coefficient for $i=0$ vanishes, because
  \[ \binom{(p-1)^2p^k-1}{p^k} \equiv \prod_{j=0}^{k+1} \binom {a_j} {b_j} \mod p \]
  where $ (p-1)^2p^k-1 = \sum_{j=0}^{k+1} a_j p^j$ and $p^k = \sum_{j=0}^k b_j p^j$ are the 
  $p$-adic expansions. However,
  \[ a_j = \begin{cases}
    p-1   & j < k \\
    0   &  j = k \\
    p-2 &  j = k+1 
  \end{cases}\]
  whereas $b_k\neq 0$. Thus the first summand in~(\ref{eq:steenrod-relation}) 
  vanishes. 
  If $\stp^i \in \mathcal{A}_p$ is indecomposable, then $i$ is a power of $p$.
  In particular, $\stp^{p^{k+1}-i}\in \mathcal{A}_p(k)$ for $i>0$. 
  For degree reasons, there is a decomposition of $\stp^{p^{k+1}-i}\in \mathcal{A}_p(k)$ which does not contain a summand 
  $A_{1} \stp^{p^k} A_{2}\stp^{p^k}\dotsm A_{j}$ with $A_{i}\in \mathcal{A}_p(k-1)$ and $j> p$. Since
  $\stp^i\in \mathcal{A}_p(k-1)$ for $i\leq p^{k -1}$, the assertion follows.
\end{proof}

\section{Some differentials}
\label{sec:some-differentials}

Recall the Goodwillie spectral sequence
\[   E_1^{-j,k} = \begin{cases}
    H^{k-j}(D_{n,j}X) &  j \geq 0 \\
    0               &  j < 0
  \end{cases}
\] 
from~(\ref{eq:goodwilliess}) for the $n$-fold loop space functor. 
In this chapter we will compute its first non-trivial differentials on 
certain classes, namely dual Dyer-Lashof operations with input from the $-1$st column. 
Furthermore we show that dual Browder operations define permanent cycles in the 
cohomology of a suspension.

\begin{Proposition}\label{difq}
  Let $1\leq n<\infty $, $r\leq n$ and $i<p-1$ be given. Then
  \[ d_i\dlc^r(x) =  0\]
  for all $x\in E_1^{-1,k+1} \cong H^{k+n}(X)$.
  Furthermore, for all $k,s\in \N$ with $k+s\equiv 0 \mod 2$ there are elements 
  $u^{k,s},v^{k,s} \in (\FF_p)^\units$ such that the equations
  \begin{subequations}
    \begin{align}
      d_{p-1}\dlc^{s(p-1)}(x)   & =  u^{k,s} \cdot \beta \stp^{\frac{k+s}{2}}(x) \label{eq:dqeven}\\
      d_{p-1}\dlc^{s(p-1)-1}(x) & =  v^{k,s} \cdot \stp^{\frac{k+s}{2}}(x) \label{eq:dqodd}
    \end{align}
  \end{subequations}
  hold for every $x \in E_1^{-1,k+1} \cong H^{k+n}(X)$.
\end{Proposition}

Proposition~\ref{difq} completely describes all differentials on classes of the form 
$\dlc^r(x)$ with $x \in E^{-1,\ast}$, since the mentioned ones are 
the only nontrivial dual Dyer-Lashof operations and higher differentials leave the second quadrant. 
In the proof given below, the formulae are reduced to $n=\infty$ and $s\in\{0,1\}$, and 
to the universal case described in the following lemma. 

\begin{Lemma}\label{lem:diff-universal}
  Let $\iota \in H^0(\Hp)$ be a generator. For every $k\in \N$ there exists an element
  $u^k\in (\FF_p)^\units$ such that in the Goodwillie spectral sequence for $\Sigma^k\Hp$ 
  and $n=\infty$ one has
  \begin{align*}
    d_{p-1}\dlc^0(\sigma^k \iota)      & =  u_k \cdot \beta \stp^{\frac k 2}(\sigma^k\iota)
    \intertext{if $k$ is even, and}
    d_{p-1}\dlc^{p-2}(\sigma^k \iota)  & =  u_k \cdot \stp^{\frac {k+1} 2}(\sigma^k\iota)
  \end{align*}
  if $k$ is odd. 
  All other differentials on these classes are zero.
\end{Lemma}                     
                      
\begin{proof}                      
  The spectral sequence for $\Sigma^k\Hp$ converges to $H^{\ast}\bigl(K(\Z/p,k)\bigr)$. 
  Suppose that $k=2i$ is even. Then $\beta \stp^i(\iota_{2i}) = 0$, where $\iota_{2i}$ is the image of
  $\iota$ in $H^{2i}\bigl(K(\Z/p,2i)\bigr)$ under the canonical map
  \[ \Sigma^\infty \Sigma^{2i} K(\Z/p,2i) \rightarrow \Hp.\] 
  Hence the element 
  $0 \neq \beta \stp^i(\sigma^{2i}\iota) \in  H^{2pi+1}(\Sigma^{2i}\Hp) = E^{-1,2pi+2}$ 
  lies in the image of some differential by Theorem~\ref{sseq}. The possibilities for this differential are 
  \[ d_r\co E_r^{-(r+1),2pi+r+1} \longrightarrow E_r^{-1,2pi+2} \]
  for varying $r$. Since the $-1$-column (a shift of the Steenrod algebra $\mathcal{A}_p$)
  consists entirely of permanent cycles 
  and the 
  differentials are derivations, all products of elements in the $-1$-column are permanent cycles.
  For degree reasons, the first $p-1$ columns solely consist of such products, whence
  the differentials are zero in this range. Consider the column $E_{p-1}^{-p,\ast}$.
  Again for degree reasons, only the dual Dyer-Lashof operations
  $\dlc^j(x)$ with $x \in E^{-1,\ast}$ can be mapped nontrivially by any differential $d_r$. 
  If $r > p-1$, then
  \[ E_1^{-(r+1),2pi+r+1} = H^{2pi}(D_{\infty,r+1}) = 0\]
  because  $2pi<2(r+1)i$ and $D_{\infty,r+1}(\Sigma^{2i}\Hp)$ is $(2i(r+1)-1)$-connected. 
  Therefore $\beta \stp^i(\sigma^{2i}\iota)$ lies in the image of
  \[ d_{p-1}\co E_{p-1}^{-p,2pi+p} \longrightarrow E_{p-1}^{-1,2pi+2}\]
  and a killer of $\beta \stp^i(\sigma^{2i}\iota)$ is a sum of elements of the form 
  $\dlc^j(x)$ for degree reasons. In order for all degrees to match, such a 
  non-zero $x$ has to be an element 
  in $E^{-1,k} =  H^{k-1}(\Sigma^{2i}\Hp)$ with $pk+j = 2pi+p$, 
  which is only possible if $k = 2i+1$ and $j=0$ by connectivity.
  Now $ H^{2i}(\Sigma^{2i}\Hp)$ is generated by $\sigma^{2i}\iota$, so for some $u \in (\FF_p)^\units$
  the equality
  \[ \beta \stp^i(\sigma^{2i}\iota) = d_{p-1}\bigl(u \dlc^0(\sigma^{2i}\iota)\bigr)\]
  holds by Proposition~\ref{qcal}. 
  With $u_{2i} = \tfrac{1}{u}$ one obtains
  \[ d_{p-1}\bigl(\mathcal Q_0(\sigma^{2i}\iota)\bigr) = u_{2i} \beta \stp^i(\sigma^{2i}\iota) \]
  which finishes the universal case for an even dimensional class. 
  The other case is similar and left to the reader.
\end{proof}

\begin{proof}[Proof of \ref{difq}]
  Let $k\in \N$ and $1\leq n<\infty$.
  The natural equivalence 
  \[ \Sigma^\infty \Omega^n K(\Z/p,k+n) \longrightarrow \Sigma^\infty \Omega^\infty\Sigma^k \Hp\]
  induces maps 
  \[ D_{n,j}(\Sigma^{-n}\Sigma^\infty K(\Z/p,k+n)) \longrightarrow D_{\infty,j}(\Sigma^k \Hp) \]
  which for $j=1$ coincide with the canonical map 
  $\Sigma^{-n}\Sigma^\infty K(\Z/p,k+n) \longrightarrow \Sigma^k \Hp$ sending 
  $\sigma^{k}\iota$ to $\sigma^{-n}\iota_{k+n}$ in cohomology. 
  Any class in $E_1^{-1,k+1} = H^k(\Sigma^{-n} \Sigma^\infty X)$ may be represented as 
  $\sigma^{-n} f^\ast(\iota_{k+n})$ for some map $f\co X\longrightarrow K(\Z/p,k+n)$.
  The naturality of the spectral sequences, the Dyer-Lashof and the Steenrod operations with 
  respect to these maps then supply equations~(\ref{eq:dqeven}) and~(\ref{eq:dqodd}) 
  in Proposition~\ref{difq} for $s=0$. 
  Proposition~\ref{epsilonCO} relates $\dlc^{(s+1)(p-1)}(x)$ in the spectral sequence for $n$ to 
  $\dlc^{s(p-1)}(\sigma x)$ in the spectral sequence for $n-1$. By hypothesis, the latter is just 
  $\beta \stp^{\frac{k+1+s}{2}}(\sigma x)$. This proves equation~(\ref{eq:dqeven})
  by induction, since the evaluation induces the identity map on the first column. 
  Equation~(\ref{eq:dqodd}) follows by the same argument.
\end{proof}

The next statement shows that all differentials vanish on dual Browder operations
originating in the $-1$st column, provided the topological space in question is a suspension.

\begin{Proposition}\label{prop:diff-browder}
  Let $1\leq n<\infty $ and $y_i \in E_1^{-1,\ast}$ for all $1 \leq i \leq m$ in the spectral sequence for 
  $\Sigma X$. Then
  \[ d_s \mathcal \brc^{n-1}(y_1 \otimes \dotsm \otimes y_m) = 0\]
  for all $s$.
\end{Proposition}

\begin{proof}
  The spectral sequence for $n=1$ converges to $H^\ast(\Omega \Sigma X)$. By a theorem of 
  Bott and Samelson \cite{BS}, the homology of $\Omega\Sigma X$ is
  -- as a graded algebra with Pontryagin product -- isomorphic to the 
  tensor algebra over the homology of $X$:
  \[ H_\ast(\Omega \Sigma X) \iso T H_\ast(X)\]  
  Thus $H^\ast(\Omega \Sigma X) \iso T H^\ast(X)$ at least as vector spaces (and coalgebras). 
  However, already 
  \[ E_1^{-j,\ast} = \shift{j}H^{\ast}(D_{1,j} \Sigma^\infty X) \iso
  \shift{j} H^\ast(X^{\wedge j}) \iso \shift{j}H^{\ast}(X)^{\otimes j}\]
  and since these isomorphisms are compatible with the filtration, 
  $E_\infty$ has to be equal to $E_1$. In particular, all differentials are zero.
  This  can be used as the start of an induction (along $n$ this time) using Proposition~\ref{epsilonCO}.
\end{proof}

The following consequence of Proposition~\ref{prop:diff-browder} 
uses the notion of desuspension index, which is introduced
in Definition~\ref{def:des-class}.

\begin{Corollary}\label{cycl}
  Let $1\leq n< \infty$. Suppose that
  $H^\ast(X)$ has desuspension index at least $n-1$. 
  Then all elements in the spectral sequence for $\Sigma X$ and $n$ 
  contained in the first $2p-1$ columns are 
  permanent cycles.
\end{Corollary}

\begin{proof}
  Proposition~\ref{prop:generated-alg} implies that all indecomposable base elements 
  in the range in question are those appearing in propositions~\ref{difq} 
  and~\ref{prop:diff-browder}, together with a copy of $\FF_p$ 
  in the $0$-column. 
  The assumption on the desuspension index of $H^\ast(X)$ implies that $d_{p-1}$ 
  vanishes on Dyer-Lashof operations and higher differentials do also
  (see the comment below Proposition~\ref{difq}).
\end{proof}

\section{Nonrealization results}
\label{sec:nonr-results}

In \cite{Schwartz:conjecture} Schwartz reduced Kuhn's Conjecture to the 
nonrealizability of certain modules over the Steenrod algebra. For $p=2$, a very readible 
account of this reduction is given at the end of \cite{Kuhn:nonrealization}. 
In order to describe these modules, consider the tensor product 
$\FF_p[t] \otimes \Lambda_{\FF_p}[s]$ of a polynomial $\FF_p$-algebra on a 
generator of degree $\lvert t\rvert =2$, and an exterior $\FF_p$-algebra 
on a generator $\lvert s \rvert=1$, with the following $\mathcal{A}_p$-action:
\begin{align*}
  \beta \cdot (t^n \otimes 1) &= 0 \\
  \beta \cdot (t^n \otimes s) &= t^{n+1} \otimes 1\\
  \stp^i \cdot (t^n \otimes 1) &=  \binom{n}{i} t^{n+i(p-1)} \otimes 1 \\
  \stp^i \cdot (t^n \otimes s)  &= \binom{n}{i} t^{n+i(p-1)} \otimes s 
\end{align*}
With this structure $\FF_p[t] \otimes \Lambda_{\FF_p}[s] \iso H^\ast \bigl(K(\FF_p,1)\bigr)$
as an $\mathcal{A}_p$-module.
We identify the submodule $\{a \otimes 1 \co a \in \FF_p[t]\}$ with $\FF_p[t]$. 
Let
\[ \Phi(k) := \langle t^{p^k} \rangle_{\mathcal{A}_p} \subseteq \FF_p[t] \]
be the submodule generated by $t^{p^k}$.  
The module structure of $\Phi(k)$ simplifies to
\begin{align*}
  \beta \cdot t^{p^j} &= 0 \\
  \stp^i \cdot t^{p^j} &= 
  \begin{cases}
    t^{p^j}     &  i = 0 \\
    t^{p^{j+1}} & i  = p^j \\
    0                     &  \text{otherwise}
  \end{cases}
\end{align*}
whence the set $\{t^{p^i} \co i \geq k\}$ is a $\FF_p$-basis for $\Phi(k)$.
Because $\stp^{p^k} \cdot t^{p^k} = t^{p^{k+1}}$,  
$\Phi(\ell)$ is a submodule of $\Phi(k)$ for $k<\ell$.

\begin{Definition}\label{def:schmodule}
  Let $k< \ell$. Then $\Phi(k,\ell) := \Phi(k) / \Phi(\ell+1)$ is the
  quotient in the category of graded $\mathcal{A}_p$-modules.
\end{Definition}

The module $\Phi(k,\ell)$ admits the residue classes of 
$\{t^{p^i} \co k \leq i \leq \ell\}$ 
as a $\FF_p$-basis. 
Only the module $\Phi(k,k+2)$ will occur in the sequel. It consists 
of three vector spaces $\FF_p \cdot t^{p^k}$, 
$\FF_p \cdot t^{p^{k+1}}$ and $\FF_p \cdot t^{p^{k+2}}$. 
The operation $\stp^{p^k}$ maps the first vector space 
isomorphically onto the second, and $\stp^{p^{k+1}}$ maps
the second vector space isomorphically onto the third. 
 
As mentioned in the introduction, our methods only 
suffice for roughly half the cases needed to provide a full proof of 
Kuhn's conjecture~\ref{conj:kuhn}. 
A precise formulation requires the notion of a {\em desuspension class\/}.

\begin{Definition}\label{def:des-class}
  Let $M$ be an unstable $\mathcal{A}_p$-module. The {\em desuspension index\/} 
  of $M$ is the largest natural number $n$ such that the shift
  $\shift{n}(M)$ is unstable. 
  If $M$ has desuspension index $n$, an element $x \in M$ with 
  \[ \begin{cases} \stp^{\frac{\lvert x \rvert -n}2} \cdot x \neq 0 & 
    \lvert x \rvert -n \even \\ 
    \beta \stp^{\frac{\lvert x \rvert -n-1}2} \cdot x \neq 0 &
    \lvert x \rvert -n \odd
  \end{cases}
  \]
  is called {\em desuspension class of even\/}
  resp.~{\em odd origin\/}. 
\end{Definition}

Every nonzero unstable $\mathcal{A}_p$-module contains 
a desuspension class.
Note that in the odd primary cohomology of a topological space 
 \[ \stp^{\frac{\lvert y \rvert}{2}}(y) = y^p\]
holds for cohomology classes of even degree only, whereas the equation
\[ \mathrm{Sq}^{\lvert y\rvert } (y) = y^2 \]
holds for every cohomology class in $H^\ast(X,\FF_2)$. Recall
Theorem~\ref{thm:main-intro} from the introduction.

\begin{Theorem}\label{thm:main}
  Let $M$ be an unstable $\mathcal{A}_p$-module 
  concentrated in degrees $[\ell,m]$. Suppose $M$
  contains a  
  desuspension class of even origin.
  If $X$ is a pointed topological space with
  \[ H^\ast(X) \cong M \otimes \Phi(k,k+2) \]
  as $\mathcal{A}_p$-modules, then $2p^k \leq (p^2-1)m + p(m-\ell)$.
\end{Theorem}

The idea of the proof of Theorem~\ref{thm:main}
is to consider the cohomology of $\Omega^n X$, where
$n$ is the desuspension index of $M$,  
by means of the spectral sequence developed above. 
In the easiest case ($n=0$), no loop space 
(hence no spectral sequence and -- as it turns out -- not even a space) 
is necessary. 
The following result proves the strong realization conjecture from \cite{Kuhn:nonrealization} in the special case $n=0$. 

\begin{Proposition} \label{prop:neqn}
  Let $M$ be an unstable $\mathcal{A}_p$-module with $\lvert M\rvert 
  \subseteq [\ell,m]$, and 
  let $x \in M$ be an element of degree $2i$
  such that $\stp^i \cdot x \neq 0$. 
  Then $M\otimes \Phi(k, k+1)$ admits an unstable 
  algebra structure only if $2(p-2)p^{k+1} \leq pm+m-p\ell$.
\end{Proposition}

The conditions on $x$ precisely say that $M$ has desuspension index $0$ and that $x$ is a desuspension class of even origin.

\begin{proof}
  Since $\lvert \stp^i \cdot x\rvert = 2pi$ 
  and $\lvert x \rvert  = 2i$, and both these elements are nonzero, one deduces 
  $0 \leq \ell \leq 2i \leq 2pi \leq m$. Let $M_1 := M \otimes (\FF_p \cdot t^{p^{k}})$ and 
  $M_2 := M \otimes (\FF_p \cdot t^{p^{k+1}})$ as sub-$\FF_p$-modules in 
  $M \otimes \Phi(k,k+1) = M_1 \oplus M_2$. 
  Then $\lvert M_1\rvert \subseteq [\ell+2p^{k}, m+2p^{k}]$ and 
  $\vert M_2\rvert  \subseteq [\ell+2p^{k+1}, m+2p^{k+1}]$.
  Assume $M\otimes \Phi(k, k+1)$ is equipped with an unstable algebra 
  structure, and that $2(p-2)p^{k+1} > pm+m-p\ell$ holds.
  This inequality implies
  in particular
  \begin{align*}
    m+2p^k  < \ell + 2p^{k+1} 
  \end{align*}           
  whence there is a gap between the sub-$\FF_p$-modules $M_1$ and $M_2$. 
  Let $b = x \otimes t^{p^k}$ 
  with $\lvert b\rvert = 2i + 2p^k$. 
  On the one hand, $\stp^{p^k+i}\cdot b = b^p\neq 0$.
  In fact, 
  \begin{align*}
    \stp^{p^k+i}\cdot b  & =  \stp^{p^k+i}\cdot \bigl(x \otimes t^{p^k}\bigr) 
   =  \sum_{j=0}^{p^k+i} \bigl(\stp^j\cdot x\bigr) \otimes \bigl(\stp^{p^k+i-j}\cdot t^{p^k}\bigr) \\
    & =  \bigl(\stp^i\cdot x\bigr) \otimes t^{p^{k+1}} + \bigl( \stp^{p^k+i}\cdot x\bigr ) \otimes t^{p^k} = \bigl( \stp^i\cdot x\bigr) \otimes t^{p^{k+1}} 
  \end{align*}
  is nonzero because neither $\stp^i\cdot x$ nor $t^{p^{k+1}}$ is zero.
  On the other hand, the sequence $b, b^2, ..., b^{p-1}, b^p$ starts in $M_1$ 
  and ends in $M_2$ and so has to pass the gap.
  Let $j$ denote the largest number such that $b^j \in M_1$. 
  The inequality $2(p-2)p^{k+1} > pm+m-p\ell$ implies
  \[ \vert b^{j+1}\rvert    =  2(j+1)i + 2(j+1)p^k 
     \leq  m+2p^k + 2i+2p^k \leq m+4p^k + \frac{m}{p}
     < \ell + 2p^{k+1} \]
  whence $b^{j+1}$ lies in the gap. Therefore also $b^p=0$.
  This contradiction shows that there can be no unstable algebra 
  structure on $M\otimes \Phi(k, k+1)$.
\end{proof}

In the case $p=2$, the first half of the 
proof of Proposition~\ref{prop:neqn} leads to
$0 \neq b^2$.
However, the sequence $b,\dotsc,b^p$ from the second half of the proof
consists only of 
$b$ and $b^2$ which lie in $M_1$ and $M_2$ respectively, and therefore $b^2$ has no 
reason to be trivial. The remedy for this is to include a third direct summand, i.e. 
study $M \otimes \Phi(k-1,k+1) = M_0 \oplus M_1 \oplus M_2$. The class 
$a=x\otimes t^{k-1} \in M_0$ then satisfies $b^2 = Sq^{2^{k-1}}(ab)$
by the Cartan formula 
and the observation that all summands but one vanish due to either instability or gaps. 
However, 
the element $ab$ lying halfway between 
$b$ and $b^2$ may be forced to fall into the gap between $M_1$ and $M_2$. 
This is exactly Kuhn's argument in \cite{Kuhn:nonrealization}. 
In the case $p$ odd, the corresponding statement is as follows:

\begin{Proposition}\label{prop:unstable-algebra}
  Let $M$ be an unstable $\mathcal{A}_p$-module with $\lvert M\rvert \subseteq [\ell,m]$ 
  and let $x \in M$ be an element of (even) degree $2i$, such that $\stp^i \cdot x \neq 0$. 
  Then $M\otimes \Phi(k, k+2)$ admits an unstable algebra structure only if 
  $2p^k \leq m$.
\end{Proposition}

\begin{proof}
  Let $a = x \otimes t^{p^k}$, $b = x \otimes t^{p^{k+1}}$ and $c = x \otimes t^{p^{k+2}}$.
  The argument just described for $p=2$ works, provided $ab$ 
  is replaced with $ab^{p-1}$ as follows. 
  The proof of Proposition~\ref{prop:neqn} applies to give $0 \neq \stp^{p^k+i}\cdot b = b^p$.
  This element may also be expressed as
  \[
    b^p=\stp^{p^k+i}\cdot b  
    =  \sum_{j=0}^{i}\bigl( \stp^j\cdot x\bigr) \otimes \bigl( \stp^{i-j}\cdot t^{p^{k+1}}\bigr) 
    =  \bigl( \stp^i\cdot x\bigr) \otimes t^{p^{k+1}}
    =  \stp^i\cdot c
  \]
  because $i<p^{k+1}$.
  Furthermore $\stp^{p^k}\cdot \bigl(a b^{p-1}\bigr) = b^p$ by the following 
  calculation:
  \begin{align*}
    \stp^{p^k}\cdot \bigl(a b^{p-1}\bigr)  & =  \sum_{j=0}^{p^k} \bigl(\stp^{p^k-j}\cdot a\bigr) \bigl( \stp^j\cdot b^{p-1}\bigr) \\
    & =  \bigl(\stp^{p^k}\cdot a\bigr) b^{p-1} + \sum_{j=1}^{p^k} \bigl(\stp^{p^k-j}\cdot a\bigr) \bigl( \stp^j\cdot b^{p-1}\bigr)
  \end{align*}
  The second summand in this expression is zero, 
  because either $\stp^{p^k-j}\cdot a$ is zero by instability, or $\stp^j\cdot b^{p-1}$ 
  is in the gap between $M_1$ and $M_2$. The equality $\stp^{p^k}\cdot a = b$ follows
  by essentially the same computation as in the proof of Proposition~\ref{prop:neqn}. 
  However, assuming $2p^k>m$, the element 
  $ab^{p-1}$ lies in the gap between $M_1$ and $M_2$, giving a contradiction. 
\end{proof}

Note that Proposition~\ref{prop:unstable-algebra} is also a corollary of Proposition~\ref{prop:neqn}
(except for $p=3$, where the conclusion is a little bit stronger).
In fact, an unstable algebra structure on $M\otimes\Phi(k, k+2)$ induces one on 
the quotient $M\otimes \Phi(k, k+1)$. 
The reason for giving a proof of Proposition~\ref{prop:unstable-algebra} is that it 
may be carried through the Goodwillie spectral sequence for higher desuspension indices.

\begin{Theorem} \label{thm:nneq}
  Let $M$ be an unstable $\mathcal{A}_p$-module with $\lvert M\rvert \subseteq [\ell,m]$, 
  and let $x \in M$ be an element of (even) degree $2i$, such that $\stp^i \cdot x \neq 0$. 
  Suppose that $X$ is a locally finite pointed cell complex with
  \[ H^\ast(X) \cong \bigl(\shift{n}M\bigr) \otimes \Phi(k, k+2)\]
  as graded $\mathcal{A}_p$-modules. Then $2p^k \leq (p^2-1)m + p(m-\ell) + (p^2-2)n+1$.
\end{Theorem}

\begin{proof}
  In order to trivialize several differentials which would otherwise become a 
  nuisance later in the proof, 
  let $X^\prime = X / X_{2i+n+2p^k-1}$ be the quotient collapsing the $2i+n+2p^k-1$-skeleton. 
  The proof employs the spectral sequence for $Y = \Sigma^{2i+1} X^\prime$ and 
  $n^\prime := n+2i+1$, which converges to the cohomology of 
  $\Omega^{n+2i+1}\Sigma^{2i+1} X^\prime$ 
  instead of $\Omega^n X^\prime$. Note that the passage from $X$ to $X^\prime$ appears already 
  in \cite{Kuhn:nonrealization}. 
  The projection $X \rightarrow X^\prime$ induces a surjection 
  $\varphi\co H^j(X^\prime) \longrightarrow   H^j(X)$ for $j=2i+n+2p^k$ and an 
  isomorphism for $j>2i+n+2p^k$.
  Let $N_j = \varphi^{-1}\bigl(M \otimes (\FF_p \cdot t^{p^{k+j}})\bigr)$ for $j\in \{0,1,2\}$ and 
  $N = \bigoplus_{j=0}^2 N_j$. Then $\shift{n}N^{\ast} =   H^\ast(Y)$ and 
  $\lvert N_0\rvert \subseteq 
  [2i+2p^k,m+2p^k]$, $\lvert N_j\rvert \subseteq [\ell+2p^{k+j},m+2p^{k+j}]$ for $j\in \{1,2\}$. 
  Let $a,b,c \in N$ be elements with $\varphi(a) = x \otimes t^{p^k}$, 
  $\varphi(b) = x \otimes t^{p^{k+1}}$ and $\varphi(c) = x \otimes t^{p^{k+2}}$. 
  Note that $b$ and $c$ are uniquely determined, because 
  $\varphi$ is an isomorphism in high degrees. 
  Suppose that 
  \begin{equation}\label{eq:ineq}
    2p^k > (p^2-1)m + p(m-\ell) + (p^2-2)n+1.
  \end{equation}
  Then in the spectral sequence for $Y$ and 
  $n^\prime$ 
  \[ E_1^{-r,\ast} =\shift{r}H^{\ast}(D_{n^\prime ,r} \Sigma^{-n^\prime}\Sigma^\infty Y) \iso
  \shift{r}H^{\ast}(D_{n^\prime ,r} \Sigma^{-n}\Sigma^\infty X^\prime )\]
  hence every element in the $-1$st column
  \[ E_1^{-1,\ast} = \shift{1-n}   H^{\ast}(X^\prime ) = \shift{1} N \]
  is a permanent cycle, since the $0$-column consists of a $\FF_p$ in degree $0$ only. 
  Thus there are (unique) elements 
  \begin{equation}\label{eq:alphabetagamma}
    \alpha \in H^{2i+2p^k}(\Omega^{n^\prime } Y),\, \beta \in H^{2i+2p^{k+1}}(\Omega^{n^\prime } Y),\,\gamma \in H^{2i+2p^{k+2}}(\Omega^{n^\prime } Y)
   \end{equation}
  representing $a \in E_\infty^{-1, 2i+2p^k+1}, b \in E_\infty^{-1,2i+2p^{k+1}+1}$ and 
  $c \in E_\infty^{-1,2i+2p^{k+2}+1}$ respectively. 
  Before proceeding to the actual argument, observe that by Proposition~\ref{prop:generated-alg}
  \[ E_1^{-r,\ast} = \sum_{\substack{(u,v,w) \in \N^3 \\ u+v+w = r}} \shift{r}\bigl(N_0^u \cdot N_1^v \cdot N_2^w\bigr)\]
  where $N_0^u \cdot N_1^v \cdot N_2^w \subseteq 
  H^\ast(D_{n^\prime ,u+v+w} \Sigma^{-n^\prime }\Sigma^\infty Y)$ is the $\FF_p$-vector space 
  generated by 'products' with $u$ factors coming from $N_0$, $v$ from $N_1$ and $w$ from $N_2$.
  Here 'products' refers to Pontryagin products, Dyer-Lashof operations (which count as $p$-fold 'products')
  and Browder operations with $s$ inputs (which count as $s$-fold 'products'). Then
  \begin{align*}
    \lvert N_0^u \cdot N_1^v \cdot N_2^w\rvert & \subseteq [2ui+(v+w)\ell +2(u+vp+wp^2)p^k, \\
      & (u+v+w)m +2(u+vp+wp^2)p^k + (u+v+w-1)(n^\prime -1)]
  \end{align*}
  with the lower bound coming from true products and the upper bound stemming from pure Browder operations.
  This concludes the preliminaries of the proof. The aim is to show that
  the element $\stp^{p^k} (\alpha\cup \beta^{p-1})$ is both zero and non-zero under the 
  assumption~(\ref{eq:ineq}). This requires a few auxiliary computations, listed
  in equations~(\ref{eq:power-alphabeta}), (\ref{eq:power-alphabetap-1}) and~(\ref{eq:power-gamma}).
  \begin{equation}\label{eq:power-alphabeta}
    \stp^{p^k}(\alpha) = \beta \quad \mathrm{and} \quad \stp^{p^{k+1}}(\beta) = \gamma
  \end{equation}
  The equations~(\ref{eq:power-alphabeta}) hold for
  $a,b,c\in E^{-1,\ast}_1$ by the
  same computation
  as in the proof of Proposition~\ref{prop:neqn}. 
  They carry over to cohomology, because the lower filtrations (i.e. the $0$-column) 
  are trivial away from degree $0$.
  \begin{equation}\label{eq:power-alphabetap-1}
    \stp^{p^k}(\alpha \cup \beta^{p-1}) = \beta^p
  \end{equation}
  In fact, the equation
  \begin{equation*}
    \stp^{p^k}(\alpha \cup \beta^{p-1})  =  
    \sum_{j_1 + \dotsm +  j_p = p^k} \stp^{j_1}(\alpha) \cup  \stp^{j_2}(\beta) \cup \dotsm 
    \cup \stp^{j_p}(\beta)
  \end{equation*}
  simplifies to the expression $ \stp^{p^k}(\alpha\cup \beta^{p-1})= \stp^{p^k}(\alpha) \cup \beta^{p-1}$
  which implies equation~(\ref{eq:power-alphabetap-1}) by~(\ref{eq:power-alphabeta}).
  One may simplify, because $\stp^{j_1} \cdot a = 0$ once $p^k > j_1 > i$, and
  because in the case $j_1 \leq i$ there exists $2\leq r\leq p$ with $j_r \geq \frac{p^k-i}{p-1}$,
  which in turn 
  implies that $\stp^{j_r} \cdot b$ is in the gap between 
  $N_1$ and $N_2$ by the numerical assumption~(\ref{eq:ineq}). The conclusion
  for $\alpha$ and $\beta$ follows as in the case of equation~(\ref{eq:power-alphabeta}).
  \begin{equation}\label{eq:power-gamma}
    \stp^i(\gamma) = \beta^p
  \end{equation}
  Here the calculation from the proof of Proposition~\ref{prop:unstable-algebra} translates 
  immediately to
  $\stp^{p^{k+1}+i}\cdot b = \stp^i\cdot c$, which extends to
  $\beta^p=\stp^{p^{k+1}+i}(\beta) = \stp^i(\gamma)$ as before. 
  The next calculation, formulated as a lemma, is the first step significantly different from the easy case. A 
  priori $\stp^i\cdot c \neq 0$ does not imply $\stp^i(\gamma) \neq 0$, because 
  $\stp^i \cdot c$ might lie in the image of some differential. 
  In fact, passing from $X$ to $Y$ is needed precisely to ensure this is not the case.

  \begin{Lemma}\label{lem:power-gamma-nonzero}
    The element $\stp^i(\gamma)$ is nonzero.
  \end{Lemma}

  \begin{proof}
    The computation from the proof of proposition~\ref{prop:neqn}
    and equation~(\ref{eq:power-gamma})
    shows $\stp^i(c) \neq 0$ as before. Thus the claim holds if 
    $\stp^i \cdot c \in E^{-1,2pi + 2p^{k+2}+1}$ is not hit by any differential
    \[ d_r\co E_r^{-(r+1), 2pi+2p^{k+2}+r} \longrightarrow E_r^{-1,2pi + 2p^{k+2}+1}.\]
    As $E_r^{-(r+1), 2pi+2p^{k+2}+r}$ is a subquotient of 
    \[E_1^{-(r+1), 2pi+2p^{k+2}+r} = H^{2pi+2p^{k+2}-1}(D_{n^\prime ,r+1}\Sigma^{-n}\Sigma^\infty X^\prime )\]
    and $D_{n^\prime ,r+1}\Sigma^{-n}\Sigma^\infty X^\prime $ 
    is $(r+1)(2i+2p^k)-1$-connected, this group is trivial once $r\geq p^2-1$.
    If $r< p^2-1$,
    \[ H^\ast(D_{n^\prime ,r+1}\Sigma^{-n}\Sigma^\infty X^\prime ) = \sum_{u+v+w = r+1} 
    N_0^u \cdot N_1^v \cdot N_2^w\]
    with no iterated Dyer-Lashof operations appearing. If 
    furthermore the inequality $u+vp+wp^2 < p^2$ holds, the top degree of $N_0^u \cdot N_1^v \cdot N_2^w$ is too low:
    \begin{align*}
      (u+v+w)m &+ 2(u+vp+wp^2)p^k + (u+v+w-1)(n^\prime -1) \\
      & \leq (p^2-1)m + 2(p^2-1)p^k+(p^2-2)(n+2i) \\
      & \leq  2p^{k+2} - 2p^k + (p^2+p-1)m + (p^2-2)n \\
      & < 2pi + 2p^{k+2}-1 
    \end{align*}
    Here the last step is precisely the numerical assumption (\ref{eq:ineq}).
    The case $u+vp+wp^2 \geq p^2$ is a little more complicated, 
    as different algebra generators do not hit $\stp^i\cdot c$
    for different reasons. 
    Elements in $E_r^{-(r+1), 2pi+2p^{k+2}+r}$ represented by (sums of) 
    products of Browder operations and elements from the $-1$-column are 
    permanent cycles by Proposition~\ref{prop:diff-browder} 
    and the fact that products of cycles are again cycles. 
    A potential killer of $\stp^i \cdot c$ therefore has to be represented by a 
    sum of products with some Dyer-Lashof operation appearing. 
    If $y \in E_1^{-s,s'}$ is the input for this operation, then
    $1 \leq s < p$ by the above discussion. Furthermore, Proposition~\ref{difq} rules 
    out $s=1$. By 
    Proposition~\ref{prop:generated-alg} it remains to consider classes 
    represented by sums of products with at least one Dyer-Lashof operation 
    applied to a Browder operation appearing. For such an element the lower bound may 
    be sharpened slightly: its degree in $N_0^u \cdot N_1^v \cdot N_2^w$ 
    is 
    \[ \geq 2ui + (v+w)l + (u+vp+wp^2)p^k + p(n^\prime -1)\] 
    because a 
    Browder operation with $s$ inputs raises degree by $(s-1)(n^\prime -1)$ in comparison to a 
    simple product, and then applying a Dyer-Lashof operation adds at least 
    $p(n^\prime -1)$. 
    This sharpened bound in the case $u+vp+wp^2 \geq p^2$ implies
    \begin{align*}
      2ui + (v+w)l + (u+vp+wp^2)p^k + p(n^\prime -1) \ \geq& \ l + 2p^{k+2} + p(n+2i) \\
      \ >& \ 2pi + 2p^{k+2} - 1
    \end{align*}
    which was the reason for suspending the space $2i+1$ times. 
    With all the differentials covered, $\stp^i(\gamma)$ is nonzero as claimed.
  \end{proof}

  Equations~(\ref{eq:power-alphabetap-1}), (\ref{eq:power-gamma}) and
  Lemma~\ref{lem:power-gamma-nonzero} imply that
  \begin{equation}\label{eq:nonzero}
    \stp^{p^k}\bigl(\alpha\cup \beta^{p-1}\bigr) \neq 0.
  \end{equation}

  However, in the current situation there is no obvious reason for $\alpha \cup \beta^{p-1}$ to be zero. 
  The gap into which one might hope a product of this type would fall is filled by its corresponding 
  Pontryagin product, at least on the $E_1$-page. However, the element 
  $\dlc^0(a) \in E_1^{-p,2pi+2p^{k+1}+p}$ is a permanent cycle by Corollary~\ref{cycl}. 
  Hence there is a class $\delta \in H^{2pi+2p^{k+1}}(\Omega^{n^\prime } Y)$ 
  representing $\dlc^0(a) \in E_\infty^{-p,2pi+2p^{k+1}+p}$ with the following property.

  \begin{Lemma}\label{lem:power-delta}
    The equation $\stp^{(p-1)p^k}(\delta)= -\alpha \cup \beta^{p-1}$ holds.
  \end{Lemma}

  \begin{proof}
    The first step is a calculation in the spectral sequence. The equality 
    \begin{align*}
      \stp^{(p-1)p^k}  \dlc^0 (a) &=  \sum_{j=0}^{(p-1)p^{k-1}} a^{0,(p-1)p^k,\lvert a\rvert }(j) 
        c^{0,(p-1)p^k,\lvert a\rvert }(j) \dlc^{2((p-1)p^k-pj)(p-1)}\stp^j(a)\\
      &+ \!\!\!\!\sum_{j=0}^{(p-1)p^{k-1}-1}\!\!\!\!\!\!b^{0,(p-1)p^k,\lvert a\rvert}(j)d^{0,(p-1)p^k,\lvert a\rvert}(j)  
        \dlc^{2((p-1)p^k-pj)(p-1)-p}\stp^j\beta (a) \\
      &+ \sum_{t \in N} \frac{1}{e_t}\stp^{t_1}(a) \star \cdots \star \stp^{t_p}(a)
    \end{align*}
    holds by the Nishida relations, Proposition~\ref{nish} (we apologize for the two-fold use of 
    the letters '$a$' and '$\beta$'). In the first and second summands, 
    $\stp^j\cdot a$ and $\stp^j\beta \cdot a$ are zero if $j > i$ by instability. 
    However, if $j\leq i$, then the index for the Dyer-Lashof operation exceeds $(n-1)(p-1)$, 
    so both summands vanish entirely. 
    The third sum reduces to the term indexed by $t := (0,p^k,...,p^k)$. In fact, if $t_2 \leq i$, 
    then $t_j>i+p^k$ for some $j$ with $2<j\leq p$, which in turn
    yields zero in the product by instability. If $i< t_2< p^k$, then $\stp^{t_2}(a)=0$, whence 
    the only element in $N$ that contributes to the sum nontrivially 
    is $(0,p^k,...,p^k)$. Thus 
    \[ \stp^{(p-1)p^k} \dlc^0 (a) = - a \ast b^{p-1}\]
    and both $\stp^{(p-1)p^k}(\delta)$ and $-\alpha \cup \beta^{p-1}$ 
    represent the same element in $E_\infty^{-p,q+p} = F^{-p}_{q}/F^{-(p-1)}_{q}$, where 
    $q := 2pi + 2p^k - 2p^{k+1}+2p^{k+2}$. Proving $F^{-(p-1)}_{q} = 0$
    it is equivalent to proving
    $E_\infty^{-r,q+r} = 0$ for all $0 \leq r < p$. However, for $0\leq r <p$, the group
    $E_1^{-r,q+r}$ lies in a gap by the following argument.
    If $w=0$, then the top degree of $N_0^u \cdot N_1^v \cdot N_2^w$ is too low,
    because
    \begin{align*}
      (u+v+w)m &+ 2(u+vp+wp^2)p^k +(u+v+w-1)(n^\prime -1) \\
       &\leq  (p-1)m + 2(p-1)p^{k+1} + (p-2)(n+2i) \\
       &\leq  2(p-1)p^{k+1} + pm-\ell+(p-2)n \\
      & <  p\ell + 2p^k + 2(p-1)p^{k+1} \leq  q.
    \end{align*}
    If $w\neq 0$, its bottom degree is too high, because 
    $2(p-1)p^k > m-\ell$ and
    \begin{align*}
      2ui+(v+w)\ell+2(u+pv+p^2w)p^k          &         \geq  \ell + 2p^{k+2} \\
      & >  m - 2(p-1)p^k + 2p^{k+2} \geq  q.
    \end{align*} 
    It follows that  $F^{-(p-1)}_{q} = 0$.
    Therefore $\stp^{(p-1)p^k}(\delta)$ and $-\alpha \cup \beta^{p-1}$ coincide 
    in $F^{-p}_{q} \subseteq   H^q(\Omega^{n^\prime } Y)$.
  \end{proof}

  Introducing this element $\delta$ goes back to
  Schwartz' paper \cite{Schwartz:conjecture}.
  The inequality~(\ref{eq:nonzero}) and Lemma~\ref{lem:power-delta} combine to give 
  $\stp^{p^k}\stp^{(p-1)p^k}(\delta) \neq 0$. However, the operation 
  $\stp^{p^k} \stp^{(p-1)p^k}$ can be decomposed using the Adem relations in the form of 
  Lemma~\ref{subalg}, which then implies the following computation.

  \begin{Lemma}\label{lem:17}
    The equation $\stp^{p^k}\stp^{(p-1)p^k}(\delta) = 0$ holds.
  \end{Lemma}

  \begin{proof}
    The idea is to find $p$ gaps in the filtration $F^{-p}$ which are large enough. 
    Unfortunately, there are only two obvious such gaps, but $\stp^{p^k}\stp^{(p-1)p^k}(\delta)$ 
    is shown to have nontrivial contribution outside $F^{-(p-1)}$ before performing its final 'jump'. 
    In $F^{-p}/F^{-(p-1)}$ we then find the required $p-2$ additional gaps. To this end define:
    \begin{align*}
      V_0  & =  \sum_{\substack{(u,v) \in \N^2, u+v\leq p \\ (0,p) \neq (u,v) \neq (1,p-1)}} N_0^u \cdot N_1^v \\
      V_1  & =  N_0 \cdot N_1^{p-1} \\
      V_2  & =  N_1^p + \sum_{\substack{(u,v,w) \in \N^3 \\ u+v+w\leq p,w\neq 0}} N_0^u \cdot N_1^v \cdot N_2^w
    \end{align*}
    and recall that shifted versions of these modules assemble to the first $p$ columns of $E_1$ in 
    such a way that
    \[ F_\ast^p = V^\infty_0 + V^\infty_1 + V^\infty_2\]
    where $V^\infty_j$ denote the subquotients of $V_j$ in $E_\infty$. 
    By definition, the degrees of elements are distributed as follows:
    \begin{align*}
      \lvert V_0\rvert & \subseteq  [2i+2p^k, (p-1)m+2(p-1)p^{k+1}+(p-2)(n^\prime -1)] \\
      \lvert V_1\rvert &\subseteq [2i\!+\!(p-1)\ell +\!(p^2-p+1)2p^{k}, pm+\!(p^2-p+1)2p^{k}\!+\!(p-1)(n^\prime -1)] \\
      \lvert V_2\rvert & \subseteq [\ell+2p^{k+2},pm+2p^{k+3}+(p-1)(n^\prime -1)]
    \end{align*}
    Hence both between $V_0$ and $V_1$ and between $V_1$ and $V_2$ there is 
    a gap spanning at least $2(p-1)p^{k-1}$. Moreover, the distance between $V_0$ and $V_2$ is 
    strictly larger than $2(p-1)p^k$.
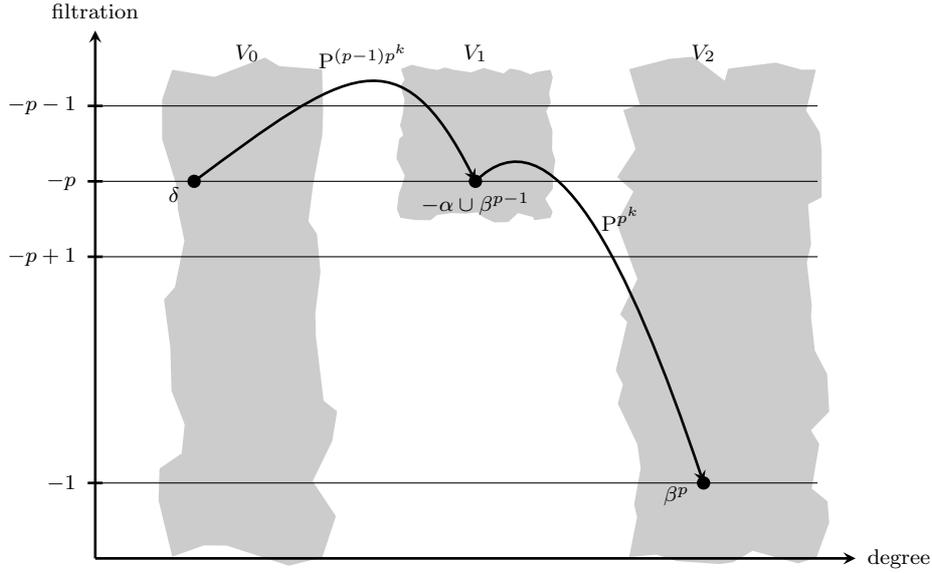
    \begin{center}
    \begin{tikzpicture}[scale=1,font=\scriptsize,line width=1pt,>=stealth]
      \filldraw[fill=black!20!white,draw=white,decorate,decoration={random steps,segment length=13pt,amplitude=6pt}] (1,0) -- (3,0) -- (3,6.5) -- (1,6.5) -- cycle;
      \filldraw[fill=black!20!white,draw=white,decorate,decoration={random steps,segment length=4pt,amplitude=2pt}] (4,4.5) -- (6,4.5) -- (6,6.5) -- (4,6.5) -- cycle;
      \filldraw[fill=black!20!white,draw=white,decorate,decoration={random steps,segment length=10pt,amplitude=5pt}] (7,0) -- (9.5,0) -- (9.5,6.5) -- (7,6.5) -- cycle ;
      \draw[black] (2,6.7) node {$V_0$};
      \draw[black] (5,6.7) node {$V_1$};
      \draw[black] (8,6.7) node {$V_2$};
      \draw[->] (0,0) -- (0,7) node[above] {filtration};
      \draw[->] (0,0) -- (10,0) node[right] {degree};
      
      \draw (0.1,1) -- (-.1,1) node[left] {$-1$};
      \draw[thin] (0.1,1) -- (9.5,1) ;
      \draw (0.1,4) -- (-.1,4) node[left] {$-p+1$};
      \draw[thin] (0.1,4) -- (9.5,4) ;
      \draw (0.1,5) -- (-.1,5) node[left] {$-p$};
      \draw[thin] (0.1,5) -- (9.5,5) ;
      \draw (0.1,6) -- (-.1,6) node[left] {$-p-1$};
      \draw[thin] (0.1,6) -- (9.5,6) ;

      \draw[fill] (1.3,5) circle (2pt) node[below=5pt,left=1pt] {$\delta$};
      \draw[fill] (5,5) circle (2pt) node[below] {$-\alpha\cup \beta^{p-1}$};
      \draw[fill] (8,1) circle (2pt) node[below=5pt,left=1pt] {$\beta^p$};

      \draw[->] (1.3,5) .. controls (3.3,6.5) and (4,7) .. (5,5) node[midway,above=1pt] {$\stp^{(p-1)p^k}$};
      \draw[->] (5,5) .. controls (6,6) and (7,4) .. (8,1) node[midway,right] {$\stp^{p^k}$};
    \end{tikzpicture}
    \end{center}

    Lemma \ref{subalg} provides a decomposition of the form 
    \[ \stp^{p^k}\stp^{(p-1)p^k} = \sum_{j=1}^r  A_{j1} \stp^{p^k} A_{j2}\stp^{p^k}\dotsm A_{je_j} \]
    with $A_{jj^\prime}\in \mathcal{A}_p(k-1)$ and $0\leq e_j\leq p$ for all $1\leq j\leq r$.
    If $\stp^{p^k} \stp^{(p-1)p^k}(\delta) \neq 0$, there exists $1\leq j\leq r$ 
    such that $\bigl(A_{j1} \stp^{p^k} A_{j2}\stp^{p^k}\dotsm A_{je_j}\bigr)(\delta) \neq 0$.
    Expressing the $A_{jj^\prime}$ via indecomposables yields (as a summand)
    a product
    \begin{equation}\label{eq:jump} \stp^{p^{\nu_1}}\dotsm \stp^{p^{\nu_t}} \quad \mathrm{with} 
       \quad 1\leq \nu_j\leq k \quad \mathrm{and}\quad 
       \bigl(\stp^{p^{\nu_1}}\dotsm \stp^{p^{\nu_t}}\bigr)(\delta)\neq 0 
    \end{equation}
    where $\nu_j=k$ for at most $p-1$ indices.
    The gap between $V_0$ and $V_2$ is too large to be passed by any single $\stp^{p^{\nu_j}}$.
    Since $\delta \in \bigl(N_0^p\bigr)^\infty \subseteq V^\infty_0$ (or rather a subquotient of $N_0$) and 
    $\stp^{p^k}\stp^{(p-1)p^k}(\delta)=\beta^p \in \bigl(N_1^p\bigr)^\infty \subseteq V^\infty_2$, there exists a
    number $1\leq s\leq t$ such that 
    $\bigl(\stp^{p^{\nu_s}}\dotsm \stp^{p^{\nu_t}}\bigr)(\delta)\in V_1$.
    The only indecomposable 
    operation appearing in~(\ref{eq:jump}) which can pass the gap between $V_1$ and $V_2$ 
    is $\stp^{p^k}$. Hence there exists $1\leq r<s$ with $\nu_r = k$. 
    In particular, $\nu_j=k$ for at most $p-2$ indices $j\geq s$.
    The same degree calculation appearing in the proof of Lemma~\ref{lem:power-delta} yields $V^\infty_1 \cap F^{-(p-1)} = \{0\}$
    on the infinity page, so 
    \[\bigl(\stp^{p^{\nu_s}}\dotsm \stp^{p^{\nu_t}}\bigr)(\delta) \neq 0 \quad \text{in} 
    \quad F^{-p}/F^{-(p-1)}.\]
    However, the image of $V_0 \oplus V_1$ in $F^{-p}/F^{-(p-1)}$ is a subquotient of 
    the module 
    \[ N_0^p \oplus \Bigl(N_0^{p-1}\cdot N_1\Bigr) \oplus 
    \Bigl(N_0^{p-2}\cdot N_1^2\Bigr) \oplus\dotsm 
     \oplus \Bigl(N_0\cdot N_1^{p-1}\Bigr) \]
    which contains at least $(p-1)$ gaps. More precisely, between any 
    $N_0^u \cdot N_1^v$ and $N_0^{u-1} \cdot N_1^{v+1}$ there is a gap spanning at least 
    $2(p-1)p^{k-1}$. Similarly, the distance between 
    $N_0^u \cdot N_1^v$ and $N_0^{u-2} \cdot N_1^{v+2}$ spans more than $2(p-1)p^k$.
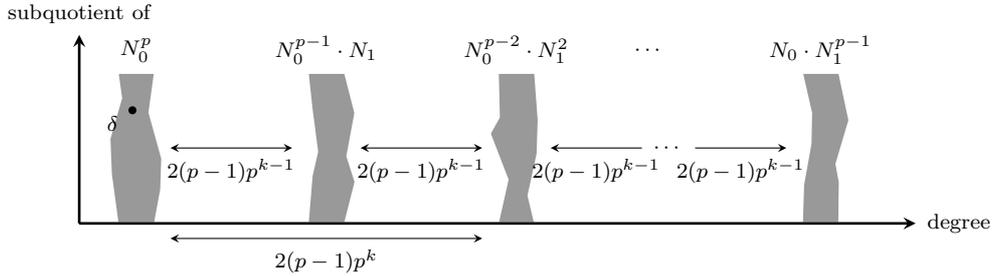
    \begin{center}
    \begin{tikzpicture}[scale=1,font=\scriptsize,line width=1pt,>=stealth]
      \filldraw[fill=black!40!white,draw=white,decorate,decoration={random steps,segment length=13pt,amplitude=4pt}] (0.5,0) -- (1,0) -- (1,2) -- (0.5,2) -- cycle;
      \draw (0.75,2.3) node {$N_0^p$};
      \filldraw[fill=black!40!white,draw=white,decorate,decoration={random steps,segment length=13pt,amplitude=4pt}] (3,0) -- (3.5,0) -- (3.5,2) -- (3,2) -- cycle;
      \draw (3.25,2.3) node {$N_0^{p-1}\cdot N_1$};
      \filldraw[fill=black!40!white,draw=white,decorate,decoration={random steps,segment length=13pt,amplitude=4pt}] (5.5,0) -- (6,0) -- (6,2) -- (5.5,2) -- cycle;
      \draw (5.75,2.3) node {$N_0^{p-2}\cdot N_1^2$};
      \filldraw[fill=black!40!white,draw=white,decorate,decoration={random steps,segment length=13pt,amplitude=4pt}] (9.5,0) -- (10,0) -- (10,2) -- (9.5,2) -- cycle;
      \draw (9.75,2.3) node {$N_0\cdot N_1^{p-1}$};

      \draw (7.5,2.3) node {$\dotsm$};
      \draw[->] (0,0) -- (0,2.5) node[above] {subquotient of};
      \draw[->] (0,0) -- (11,0) node[right] {degree};
      \draw[<->,thin] (1.2,1) -- (2.8,1)  node[midway,below] {$2(p-1)p^{k-1}$};
      \draw[<->,thin] (3.7,1) -- (5.3,1)node[midway,below] {$ 2(p-1)p^{k-1}$};
      \draw[<->,thin] (1.2,-0.2)--(5.3,-0.2) node[midway,below]{$2(p-1)p^{k}$};
      \draw[<-,thin] (6.2,1) -- (7.4,1)  node[midway,below] {$2(p-1)p^{k-1}$};
      \draw[->,thin] (8.1,1) -- (9.3,1)  node[midway,below] {$2(p-1)p^{k-1}$};
      \draw (7.75,1) node {$\dotsm$};
      \draw[fill] (0.7,1.5) circle (1pt) node[below=5pt,left=1pt] {$\delta$};
    \end{tikzpicture}
    \end{center}

    Since no indecomposable operation in~(\ref{eq:jump})
    can pass across two gaps at once, the product 
    $\bigl(\stp^{p^{\nu_s}}\dotsm \stp^{p^{\nu_t}}\bigr)(\delta)$ 
    has to factor through a trivial module 
    in $F^{-p}/F^{-(p-1)}$. It follows that $\stp^{p^k}\stp^{(p-1)p^k}(\delta) = 0$.
  \end{proof}

  The inequality~(\ref{eq:nonzero}), Lemma~\ref{lem:power-delta} and Lemma~\ref{lem:17}
  provide that the element
  $\stp^{p^k}\stp^{(p-1)p^k}(\delta) = - \stp^k(\alpha \cup \beta^{p-1})$ is both 
  zero and nonzero. Hence the assumption~(\ref{eq:ineq}) is wrong, which proves
  Theorem~\ref{thm:nneq}.
\end{proof}

The theorem as stated in the beginning can now be obtained by a simple reindexing using some obvious estimates. 
Note that $M$ has a desuspension class of even origin 
if the Bockstein acts trivially on $M$. 
Following Kuhn's observations in \cite{Kuhn:realizing}, we recover a theorem of his:

\begin{Theorem}\label{thm:finite-bockstein}
  If $H^\ast(X)$ is finitely generated over the 
  Steenrod algebra $\mathcal{A}_p$ and 
  the Bockstein acts trivially on it in high degrees, then $H^\ast(X)$ is finite.
\end{Theorem}

\begin{proof}  The proof proceeds as in \cite{Kuhn:nonrealization}. 
  Let $\bar{T}$ be the reduced Lannes functor and 
  $\Delta (X)=\mathrm{Map}(B\Z/p,X)/X$ where $X$ is embedded as the space
  of constant maps. Iterated applications of 
  $\Delta$  produce a space $Y$ which still has finitely generated, 
  but infinite dimensional, cohomology and whose cohomology satisfies 
  $\bar{T}^2 \bigl(H^\ast (Y)\bigr)=0$. Structure theorems of Kuhn and Schwartz 
  imply that in large degrees, $H^\ast (Y)$ is isomorphic to the module 
  $M\otimes \Phi (k,\infty )$ with $M\cong \bar{T}H^\ast (Y)$. 
  By a result of Winstead recorded in 
  \cite[Theorem 1.3]{Kuhn:realizing}, the Bockstein still acts trivial on 
  $M$ in large degrees. For large $k$ one can produce a subquotient space $Z$ with 
  $H^\ast (Z)\cong M\otimes \Phi (k,k+2)$, which contradicts Theorem~\ref{thm:nneq}.
\end{proof}

The case of a desuspension class of odd origin remains unsolved.
The operation $\beta \stp^i $ has an interpretation in terms of $p$-fold 
symmetric Massey products \[\beta \stp^i (x)=\langle x,\dots ,x\rangle_s\]  
by work of Kraines \cite{Kraines}. A similar formula holds for 
$Q_{p-2}$ by a result of Kochman \cite{Kochman}. 
Consequently, one may try to replace $p$-fold cup products 
in the arguments for the even case by 
$p$-fold Massey products. This attempt 
works if the module $M$ does not desuspend at all 
(the case $n=0$), via a Cartan formula for Massey products. 
If $n>0$, one may try to replace $\dlc^0$ by $\dlc^{p-2}$
and use the interpretation in terms of Massey products mentioned above. 
Currently we do not know how to make this work.
Another strategy is to intoduce an extra looping to reduce to the even case.
This strategy already appears in the $K$-theoretical Adams-Atiyah proof of the Hopf
invariant one theorem at odd primes, and it was also used in \cite{Schwartz:conjecture}.
However, in that situation the element $d_{p-1}\dlc^{0}(\alpha)$ is possibly nonzero.
There are several ways to deal with this.
One approach is to bring the projective plane into play, as was done 
in \cite{Schwartz:conjecture}.
As noted in \cite{Schwartz:Errata}, it is unclear whether 
$\dlc^{0}(\alpha)$ still survives in the spectral sequence
for the projective plane.
Another approach
is to switch to the homotopy fibre $F$ of $\beta \stp^i(\alpha)$. 
Here the class corresponding to $\dlc^{0}(\alpha)$
surely survives, but an extension problem 
related to the $H$-deviation of the splitting map for $\Omega F$ as 
a product of
the loop space on  $X$ and the appropriate Eilenberg-MacLane space 
prevents the argument to go through.
In both approaches, one uses the fact that $\beta \stp^i$ loops down to zero.
Harper shows that the problems in both approaches are closely related 
in \cite{Harper:secondary}.

\appendix

\section{Homology of free algebras over $E_n$-operads}
\label{sec:homol-free-algebr}

Our aim in this section is to compute the (co)homology of the 
free algebra on a spectrum $E$ over a suitable $E_n$-algebra. 
The case of an $E_{\infty}$-operad was already treated by McClure in \cite{B.M.M.S.}.
The homology of the free $\mathcal{C}_n$-algebra $\mathcal{C}_n X$ on a space 
$X$  has been computed
by Fred Cohen in \cite{C.L.M.}. More precisely, he defined a category of so called 
allowable $AR_{n-1}\Lambda_{n-1}$-Hopf algebras and constructed  
(among other things) a free functor $W_{n-1}$ from a category of 
coalgebras to this category of allowable 
$AR_{n-1}\Lambda_{n-1}$-Hopf algebras. 
Cohen then proved that the homology $H_{\ast}\mathcal{C}_n X$ of the free 
algebra over the operad of little $n$-cubes on a space $X$ 
is isomorphic to $W_{n-1}H_{\ast}X$ as an $AR_{n-1}\Lambda_{n-1}$-Hopf algebra.
Besides the Hopf algebra structure, the action of the Steenrod reduced powers, 
the Dyer-Lashof  and the Browder operations
are encoded in the $AR_{n-1}\Lambda_{n-1}$-Hopf algebra structure. 
The precise definition
is given on the first few pages of \cite{C.L.M.}.

Building on Cohen's calculation, most of the axioms where shown to hold  
for the homology of
a suitable $E_n$-spectrum (or even an $H_n$-spectrum) by Steinberger in 
\cite{B.M.M.S.}. The part of the structure which was not studied there 
is the comultiplication and its interaction with the operational structure. 
In the papers \cite{A.K.} and \cite{Kuhn:diagonal}, Kuhn
studied a natural coproduct on the homology of a free operad algebra on a 
spectrum, and we rely very much 
on these results. The unstable case of an $E_n$-operad was treated already in 
\cite{C.L.M.}, and Strickland and Turner considered the
$E_{\infty}$ case in \cite{Strickland:symmetric} and \cite{S.T.}.

Let $\mathcal{C}$ be a topological operad augmented over the operad 
$\mathcal{L}$ of linear isometries, let $E$ be a spectrum and 
$E_+ =E\vee \mathbb{S}$, where $\mathbb{S}$ is the sphere spectrum. 
Then the free $\mathcal{C}$-algebra on $E_+$ is defined in \cite{L.M.S.}. 
Let $\ep_k$ denote the $k$-th extended powers with respect to $\mathcal{C}$. Write 
$\Delta \co E_+\to E_+\vee E_+$ for the stable pinch map and consider the 
map $\eta$ given by the composition
\[\mathcal{C}E_+\!\xrightarrow{\mathcal{C}\Delta}\mathcal{C}(E_+\vee E_+)
\simeq \!\!\bigvee\limits_{k=0}^{\infty}\ep_k (E_+\vee E_+)\xrightarrow{\rho}
\!\!\bigvee\limits_{k=0}^{\infty}\bigvee\limits_{i+j=k}\!\ep_i E_+\wedge 
\ep_j E_+ \!
\simeq \mathcal{C}E_+\wedge \mathcal{C}E_+\]
where $\rho$ is induced by the degeneracies of the operad 
(see \cite[p.~12]{May:geometry}). Roughly, it associates to a 
labeled ``configuration'' all partitions of the configuration. 
Following Kuhn, we choose 
\[\psi_\ast\co H_{\ast}\mathcal{C}\Sigma^{\infty}X\to H_{\ast}\mathcal{C}
\Sigma^{\infty}X\otimes H_{\ast}\mathcal{C}\Sigma^{\infty}X\] 
as coproduct. 
The homomorphism induced by $\rho$ on homology can be described by 
suitable transfer homomorphisms. 
It is part of the main result in \cite{Kuhn:diagonal} that, for a co-$H$-space $X$, 
$\psi$
coincides with the usual coproduct. Here the isomorphism 
\[\mathcal{C}\Sigma^{\infty}X\cong \Sigma^{\infty}\mathcal{C}X\]
enters. Since $\mathcal{C}_n$ is not  augmented over the operad $\mathcal{L}$
of linear isometries, the free $\mathcal{C}_n$-algebra $\mathcal{C}_n E$ on a 
spectrum $E$ in the sense of \cite[Chapter VII]{L.M.S.} cannot be formed.
However, a cofibrant model 
$\bar{\mathcal{C}} \xrightarrow{\simeq} \mathcal{C}_n$ in the 
model category of topological operads \cite{B.M.} is augmented over $\mathcal{L}$. 
Alternatively, one can choose the operad $\mathcal{C}_n \times \mathcal{L}$ 
instead.
Moreover, functorial cellular approximation leads to a cellular operad,
still denoted $\bar{\mathcal{C}}$, because it commutes with products.

\begin{Remark}\label{rem:operad-ring-spectra}
  There is a notion of {\em naive\/}  (MIT) $\mathcal{C}$-ring spectra, 
  for which one does not need to assume that the operad 
  $\mathcal{C}$ is augmented over $\mathcal{L}$.
  No operadic internalization of external smash products is needed, 
  instead one uses the
  symmetric monoidal smash product in $S$-modules and 
  the fact that the model category of $S$-modules
  is tensored over topological spaces.
  However, we have to rely on certain results in \cite{L.M.S.}
  and there is no comparison between these two 
  notions of operad ring spectra in the literature as far as we know
  (but see \cite{May:precisely}).
\end{Remark}

We remind the reader how $\ep_{k}\Sigma^d X$ can be viewed as a relative
Thom complex in the case where $X$ is a pointed space. Consider the bundle 
\[p\co \bar{\mathcal{C}}_{k}\times_{\Sigma_k}(\mathbb{R}^d)^k \to 
\bar{\mathcal{C}}_{k}/\Sigma_k = \bar{\mathcal{B}}_{k}\]
and let $p_X$ the pullback along the map
\[ q\co\bar{\mathcal{C}}_k\times_{\Sigma_k}X^k \to \bar{\mathcal{B}}_{k}.\]
Then $\ep_{k}\Sigma^d X$ is the quotient of Thom complexes $T(p_X)/Tp_{\bullet}$,
where $\bullet$ is the base point in $X$.
Extended powers of spectra do not commute with suspension. 
However, there is a Thom isomorphism $\Phi^{-1}$ for a spectrum $E$ and 
$d$ even whose inverse
\[\Phi \co H_{\ast}\ep_{k}\Sigma^{d} E \to H_{\ast}\ep_{k} E \] 
is more convenient for us. In the $E_\infty$-case, such a Thom isomorphism is 
provided in 
\cite[Theorem VII.3.3] {B.M.M.S.} (see also \cite[Chapter VII]{C.L.M.}), 
and the proof generalizes. The homomorphism $\Phi$ is induced by the composite 
\[H_{\ast}\ep_{k}\Sigma^d E\xrightarrow{\delta_{\ast}}
H_{\ast}\ep_{k}S^d \wedge \ep_{k}E \cong 
H_{\ast} \ep_{k}S^d \otimes H_{\ast}\ep_{k}E \xrightarrow{\epsilon\otimes\id}
H_{\ast}\ep_{k}E\] 
where $\delta\co \ep_{k}\Sigma^d E\to \ep_{k}S^d \wedge \ep_{k}E $ 
is induced by the diagonal in $\bar{\mathcal{C}}$ and 
$\epsilon$ is evaluation at the orientation class 
$e_0 \otimes \iota^k\in H^{\ast}\ep_{k}S^d$. Here $\iota$ is the fundamental 
class of $S^d$ and $e_0 =1$ is the image of the bottom class of the standard 
free acyclic $\mathbb{Z}/p$-complex $W$ in 
$H^{0}\bar{\mathcal{C}}_k$ -- see \cite[p.403]{L.M.S.} and 
\cite[p.244-247]{B.M.M.S.} for details. 
The isomorphism $\Phi$ is the colimit of a compatible system of 
isomorphisms $\Phi_s$  of (inverse) Thom isomorphisms for a 
system of bundles defined by the augmentation to $\mathcal{L}$, followed 
by suitable desuspension isomorphisms.
Snaith's splitting assembles these isomorphisms $\Phi$ for various $k$
to an isomorphism 
\[\Phi\co H_{\ast}\mathcal{C}\Sigma^d E_+ \to H_{\ast}\mathcal{C}E_+ .\]

\begin{Lemma}\label{lem:thom-iso-coalg}
  The isomorphism \[\Phi\co H_{\ast}\bar{\mathcal{C}}\Sigma^d E_+ \to 
  H_{\ast}\bar{\mathcal{C}}E_+.\] is compatible with the comultiplication $\psi$.
\end{Lemma}

\begin{proof}
  This follows from naturality of $\Phi$ applied to 
  $\rho\mathcal{C}(\Delta)$, just as in \cite[7.1]{Strickland:symmetric}.
\end{proof}

The next lemma shows that $\Phi$ is also compatible with Browder operations, 
Dyer-Lashof operations and the Pontryagin product. A version of it is 
proved for cyclic powers in 
\cite[Lemma VIII.3.3.]{L.M.S.}. 
The arguments given there generalise once one replaces the 
standard acyclic $\mathbb{Z/}p$ complex $W$ \cite{May:steenrod} 
by the spaces in $\bar{\mathcal{C}}$.
Before we state the lemma, recall that \cite[VIII 2.9.]{L.M.S.} supplies
an isomorphism 
\[C_{\ast}( \ep_{k}E) \cong  C_{\ast}(\bar{\mathcal{C}}_{k})
\otimes_{\Sigma_k}C_{\ast}(E)^{\otimes k} \]
for a cellular operad $\mathcal{C}$ and a 
cellular spectrum $E$,
where $C_{\ast}$ denotes the cellular chain complex. 
The chain complex $C_{\ast}(E)^{\otimes k}$ is 
equivariantly chain homotopy equivalent to $H_{\ast}(E)^{\otimes k}$. 
As a consequence, any class in $H_{\ast}\ep_{k}E$ 
may be represented as 
$[c\otimes x_1 \otimes \dots \otimes x_k]$ with 
$c\in C_{\ast}(\bar{C}_{k})$ and $[x_i]\in H_{\ast}E$.

\begin{Lemma}\label{lem:thom-iso-formula}
  The equality
  \[\Phi \bigl([ c\otimes \sigma^{d}x_1 \otimes \dots \otimes \sigma^{d}x_k ]\bigr)=
	 [ c\otimes x_1 \otimes \dots x_k ]\] holds.
\end{Lemma}

\begin{proof}
  Consider a class $z\in H_{\ast}\Sigma^d \ep_k E$ and represent it as 
  \[z=[c\otimes \sigma^d x_1 \otimes \dots \otimes \sigma^d x_k].\] 
  Express 
  $\widetilde{\Delta} c = 
  \sum\limits_{i=1}c_i^\prime\otimes c_i^{\prime\prime}$,
  where $\widetilde{\Delta}$ is a diagonal approximation for 
  $\bar{\mathcal{C}}_k$.
  By definition of $\widetilde{\Delta}$, the class  $z$ is mapped to 
  \[\sum\limits_{i=1}(c_i^\prime\otimes\sigma\otimes \dotsm \otimes \sigma ) 
  \otimes (c_i^{\prime\prime}\otimes x_1 \otimes \dotsm \otimes x_k )\] 
  where $\sigma$ is interpreted as the generator of $H_{\ast}S^d$. 
  Evaluating the orientation class  $1\otimes i^k \otimes \id$ on that 
  cycle yields $[ c\otimes x_1 \otimes \dots x_k ]$, which finishes
  the proof.
\end{proof}

\begin{Corollary}\label{cor:thom-iso-everything}
  The Thom isomorphism $\Phi$ commutes with 
  Browder operations, Dyer-Lashof operations $Q_i$ and the Pontryagin product.
\end{Corollary}

\begin{proof}
  The operations are defined to be the image of a class 
  $[ c\otimes x_1 \otimes \dots x_k ]$ under the algebra structure map,
  where $c$ is either one of Cohen's classes $e_i$, or
  the fundamental class $\iota$ for $\bar{\mathcal{C}}_2\simeq S^{n-1}$, 
  or a class in $H_0 \bar{\mathcal{C}}_2$ \cite[p.~66]{May:geometry}. 
  Moreover, it suffices to consider free $\bar{\mathcal{C}}$-algebras. 
  The assertion then follows directly from Lemma~\ref{lem:thom-iso-formula}. 
\end{proof}

The coproduct $\psi$ on the homology of a $\bar{\mathcal{C}}$-algebra on a 
space $X$ is related to the other operations by diagonal 
Cartan formulas \cite{C.L.M.}. These read:
\begin{align}
  \psi Q^k (x) & = \sum\limits_{i+j=k}\sum\limits_r Q^i (x^{\prime}_r )\otimes 
  Q^j (x^{\prime\prime}_r) \\ 
  \psi\lambda_{n-1}(x,y)&=\sum\limits_{r,s} 
  (-1)^{(n-1)\lvert x^{\prime}_r\rvert+\lvert x^{\prime\prime}_r\rvert 
    \lvert y^{\prime}_s\rvert}\lambda_{n-1}(x^{\prime}_r y^{\prime}_s )
  \otimes x_r^{\prime\prime}y_s^{\prime\prime} \\
  & + \sum\limits_{r,s}(-1)^{(n-1)\lvert y_s^{\prime}\rvert+
    \lvert x_r^{\prime\prime}\rvert \lvert y_{s}^{\prime}\rvert}x_r^{\prime}
  y_{s}^{\prime}\otimes \lambda_{n-1}(x_{r}^{\prime\prime},y_{s}^{\prime\prime})\\
  \mathrm{if} & \quad\psi (x)=\sum\limits_{r}x_r^{\prime}\otimes x^{\prime\prime}_r  
  \quad \mathrm{and} \quad 
  \psi (y)=\sum\limits_{s} y_{s}^{\prime}\otimes y^{\prime\prime}_{s}
\end{align}
The nonlinear top operation $\xi$ satisfies a diagonal formula 
only up to an error term consisting of $p$-fold Browder operations in the 
elements $x^{\prime},x^{\prime\prime}$. 
This fact was overlooked in \cite{C.L.M.} and corrected 
by Wellington \cite{Wellington}.
The diagonal formula for the operation $\zeta$ follows from the formulae
for $\xi$ and $\lambda_{n-1}$, together with the relation 
$\zeta (x) = \beta \xi (x) - \mathrm{ad}^{p-1}(x)(\beta x)$.

\begin{Corollary}\label{cor:cartan}
  The diagonal Cartan formulas hold in  $H_{\ast}\bar{\mathcal{C}} E$.
\end{Corollary}

\begin{proof}
  This assertion is proven in \cite{C.L.M.}
  for a suspension spectrum $\Sigma^{\infty}X$. 
  By \cite[I.4.7]{L.M.S.} there is a natural isomorphism  
  $E\cong  \colim_m \Sigma^{-m}\Sigma^{\infty}E_m$.
  Since extended powers commute with colimits by~\cite[Prop.~I.1.2]{B.M.M.S.} 
  it suffices to prove the assertion for a shifted
  suspension spectrum $\Sigma^{-m}\Sigma^{\infty}X $,
  where $X$ is a pointed cell complex.  
  Furthermore, one can restrict to the case where $X$ is connected and $m$
  is even because of the natural isomorphism 
  \[\Sigma^{-m}\Sigma^{\infty}X \cong \Sigma^{-(m+1)}\Sigma^{\infty}\Sigma X \]
  (see \cite[I.4.2]{L.M.S.}). Then the assertion follows from
  Lemma~\ref{lem:thom-iso-formula}. 
\end{proof}

The following is now immediate:

\begin{Proposition}\label{prop:thom-iso}
  Let $E=\Sigma^{-m}\Sigma^{\infty}X $ for a connected cell
  complex $X$ and $m$ even. 
  Then the composition
  \[W_{n-1}H_{\ast}\Sigma^m E\xrightarrow{\eta}H_{\ast}\bar{C}_{n-1}\Sigma^m E
  \xrightarrow{\Phi^{\frac{m}{2}}} H_{\ast}\bar{C}_{n-1} E\]
  is an isomorphism of $AR_{n-1}\Lambda_{n-1}$-Hopf algebras.
\end{Proposition}

Note that the isomorphism in Proposition~\ref{prop:thom-iso} does not
preserve the degree. As an algebra, $W_{n-1}H_{\ast}\Sigma^m E$ is 
isomorphic to a free commutative graded algebra  $\Lambda M_{n-1}$. 
The vector space $M_{n-1}$ has a basis which consists of classes 
$Q^I y$, where $I$ runs through certain admissible sequences and 
$y$ runs through basic products in the $\lambda_{n-1}$-algebra 
underlying $W_{n-1}H_{\ast}E$ (see \cite[p.227]{C.L.M.} for precise definitions). 
This information and Proposition~\ref{prop:thom-iso} imply the next result,
which is the raison d'\^{e}tre of this appendix.

\begin{Corollary}\label{cor:div-power}
  Let $E=\Sigma^{-m}\Sigma^{\infty}X $ for a connected cell
  complex $X$ and $m$ even. 
  As an algebra, $H^{\ast}\bar{C}_{n-1} E$ is the 
  free algebra $\Gamma \shift{k}{M}_{n-1}^{\ast}$ with divided powers 
  on the dual of a shifted copy $\shift{k}{M}_{n-1}$ of $M_{n-1}$.
\end{Corollary}

\begin{proof}
  The assertion follows from Proposition~\ref{prop:thom-iso} and 
  \cite[3.13]{Wellington}. 
\end{proof}

\section{Nishida relations}
\label{sec:nishida-relations}

This final section presents the proofs of
Proposition~\ref{qcal} and 
the Nishida relations in cohomology, Proposition~\ref{nish}.
The standard Nishida relations in homology given
(\cite[p. 209, Theorem 9.4]{May:steenrod}) are a major ingredient
in this proof. 
The modified version of the two top operations from \cite{C.L.M.} 
is not required here.

\begin{Proposition}[Nishida relations]\label{hnis}
  Let $s\in \N$ and $0\leq r<(p-1)(n-1)$. Then the equality
  \begin{align*}
    \stp_s\bigl( Q_r(x)\bigr)  & =  
    \sum_{i=0}^{\bigl\lfloor \frac{s}{p} \bigr\rfloor} 
    a_{r,s,\lvert x\rvert}(i)  Q_{r+2(pi-s)(p-1)} \stp_i(x) \\
    & \quad +  \delta_r\sum_{i=0}^{\bigl\lfloor \frac{s-1}{p}\bigr\rfloor} 
    b_{r,s,\lvert x\rvert}(i) Q_{r+p+2(pi-s)(p-1)} \beta \stp_i(x)
  \end{align*}
  holds, where
  \begin{align*}
    \delta_r  & = \begin{cases} 0 & r \equiv 0 \mod 2 \\ 1 & r \equiv 1 \mod 2
    \end{cases} \\
    a_{r,s,\lvert x\rvert}(i)  & =  \binom{s-ps+\bigl\lfloor\frac{r}{2}\bigr\rfloor + \frac{p-1}{2}\lvert x \rvert}{s-pi}\\
    b_{r,s,\lvert x\rvert}(i)  & =  (-1)^{\frac{p-1}{2}\lvert x \rvert +1} 
    \bigl(\tfrac{p-1}{2}\bigr)!\binom{s-ps-1+\bigl\lfloor\frac{r+1}{2}\bigr\rfloor +
      \frac{p-1}{2}\lvert x \rvert }{s-pi-1}
  \end{align*}
\end{Proposition}

We turn now to the proof of Proposition~\ref{nish}.

\begin{proof}
  Let $\Sigma_p$ act on $\N^p$ by permuting coordinates, set
  \[ N: = \bigl\{n \in \N^p \co n_i \leq n_{i+1}, \sum n_i = s, \exists j\co n_j < n_{j+1}\bigr\}\subset \N^p \]
  and choose $n\in N$.
  To justify that
  the residue class $e_n$ of the order of the isotropy
  group is invertible in $\FF_p$, observe that
  the prime $p$ divides $e_n$ if and only if $n_i = n_j$ for all $i,j$
  for the following reason.
  In fact, if all coordinates of $n\in \N^p$ are equal, 
  the isotropy group of $n$ has order $p!$, being $\Sigma_p$. 
  If $p$ divides $e_n$,
  there exists an element $\tau$ of order $p$ in the isotropy group of $n$
  by Cauchy's theorem from group theory. The only elements of order $p$ in 
  $\Sigma_p$ are $p$-cycles. Thus $\tau
  \cdot n = n$ implies $n_i = n_j$ for all $i,j$.
  It follows that $\tfrac{1}{e_n}$ is well defined.

  The proof is basically a computation of pairings with
  generating elements of the corresponding homology, together
  with references to Proposition~\ref{hnis}.
  Pairing equation~(\ref{eq:nish1})
  with the Dyer-Lashof operations yields as left hand side 
  \[ \bigl\langle  \stp^s  \dlc^r(x) ,  Q_j(y) \bigr\rangle =
  \bigl\langle   \dlc^r(x) , \stp_s  Q_j(y) \bigr\rangle \]
  whereas the right hand side will be considered separately for each of
  the three summands (with the third only appearing if $r=0$). 
  The first summand is
  \[\sum_{i=0}^{\bigl\lfloor \frac{s}{p}\bigr\rfloor} a^{r,s,\lvert x\rvert}(i) 
  c^{r,s,\lvert x\rvert}(i)  \dlc^{r+2(s-pi)(p-1)} \stp^i(x).\]
  On the cohomological side, the pairing $\bigl\langle
  \dlc^{r+2(s-pi)(p-1)} \stp^i(x) ,  Q_j(y) \bigr\rangle$ is
  non-zero only if $j = r+2(s-pi)(p-1)$ and either
  \begin{equation*}
    2(p-1)  \mid (p-1)\lvert y\rvert+j
    \quad \mathrm{or} \quad
    2(p-1)  \mid (p-1)\lvert y\rvert +j+1.
  \end{equation*}
  For the purpose of this proof, such an index $j$ will be called \emph{good}.
  Let $q := \frac{r-j+2s(p-1)}{2(p-1)p}$, then
  \begin{align*}
    \bigl\langle \sum_{i=0}^{\bigl\lfloor  \frac{s}{p} \bigr\rfloor}
    a^{r,s,\lvert x\rvert}(i) c^{r,s,\lvert x\rvert}(i)
    & \dlc^{r+2(s-pi)(p-1)} \stp^i(x) ,Q_j(y) \bigr\rangle \\
    &= \begin{cases}
      a^{r,s,\lvert x\rvert}(q)\gamma(\lvert x \rvert,r)\bigl\langle\stp^q(x),y\bigr\rangle 
      & q \in \N, j \ \mathrm{good} \\
      0  & \mathrm{else}
    \end{cases}
  \end{align*}
  On the homological side, $\bigl\langle  \dlc^r(x) ,Q_{j+2(pi-s)(p-1)}  \stp_i(x) \bigr\rangle$ 
  is nonzero only if $j = r+2(s-pi)(p-1)$ and either
  \begin{align*}
    2(p-1) & \mid j+2(pi-s)(p-1) + (p-1)(\lvert y \rvert -2(p-1)i)\\
    \intertext{or}
    2(p-1) & \mid j+2(pi-s)(p-1) + (p-1)(\lvert y \rvert -2(p-1)i)+1
  \end{align*}
  which reduces to $j$ being good. As above
  \begin{align*}
    \bigl\langle  \dlc^r(x) , \sum_{i=0}^{\bigl\lfloor \frac s p \bigr\rfloor} 
    a_{j,s,\lvert y\rvert}(i) \bigr. &  Q_{j+2(pi-s)(p-1)} \stp_i(y) \bigr\rangle \\
    &= \begin{cases}
      a_{j,s,\lvert y \rvert}(q) \gamma(\lvert x \rvert ,r) \bigl\langle x, \stp_q(y) \bigr\rangle 
      & q \in \N, j \ \mathrm{good} \\
      0  & \mathrm{else}.
    \end{cases}
  \end{align*}
  Both these terms are nonzero only if $\lvert \stp^q(x)\rvert = \lvert y \rvert $, and then a
  simple calculation shows $a_{j,s,\lvert y\rvert}(q) = a^{r,s,\lvert x\rvert}(q)$. 
  The equality
  \begin{multline}\label{eq:dl1}
    \Bigl\langle \sum_{i=0}^{\bigl\lfloor \frac s p \bigr\rfloor} a^{r,s,\lvert x\rvert}(i)c^{r,s,\lvert x\rvert}(i)
    \dlc^{r+2(s-pi)(p-1)} \stp^i(x) ,  Q_j(y)\Bigr\rangle = \\
    \bigl\langle  \dlc^r(x) , \sum_{i=0}^{\bigl\lfloor \frac s p\bigr\rfloor} 
    a_{j,s,\lvert y \rvert}(i)  Q_{j+2(pi-s)(p-1)}  \stp_i(y) \bigr\rangle
  \end{multline}
  follows. 
  The second summand is 
  \[ \delta^r\sum_{i=0}^{\bigl\lfloor \frac{s-1}{p} \bigr\rfloor}b^{r,s,\lvert x\rvert }(i) 
  d^{r,s,\lvert x\rvert}(i)  \dlc^{r+2(s-pi)(p-1)-p}  \stp^i\beta (x). \]
  On the cohomological side, the term $\delta^r \bigl\langle 
  \dlc^{r-p+2(s-pi)(p-1)} \stp^i\beta (x), Q_j(y)\bigr\rangle$ is
  nonzero only if $r$ is even, $j = r-p+2(s-pi)(p-1)$ and either
  \begin{equation*}
  2(p-1) \mid (p-1)\lvert y \rvert +j \quad \mathrm{or} \quad
  2(p-1) \mid (p-1)\lvert y \rvert +j+1.
  \end{equation*}
  Note that $j$ is odd if $r$ is even, so $2(p-1) \nmid (p-1)\lvert y \rvert +j$,
  eliminating the first of these two cases. Let $q := \frac{r-j-p+2(p-1)s}{2(p-1)p}$ 
  and call an index $j$ {\em bad\/} if $j = r-p+2(s-pi)(p-1)$ and 
  $2(p-1) \mid (p-1)\lvert y \rvert +j+1$. Then the equality
  \begin{multline}
    \bigl\langle  \delta^r \sum_{i=0}^{\bigl\lfloor \frac{s-1}{p}  \bigl\rfloor}
    b^{r,s,\lvert x\rvert}(i)d^{r,s,\lvert x\rvert}(i)   \dlc^{r-p+2(s-pi)(p-1)}  \stp^i\beta(x),
    Q_j(y) \bigr\rangle \\
     =  \begin{cases}
      \delta^r b^{r,s,\lvert x\rvert}(q)\gamma(\lvert x\rvert,r)\bigl\langle\stp^q\beta(x),y \bigr\rangle
     & q \in \N, j\ \mathrm{bad}\\
     0  & \mathrm{else}
    \end{cases}
  \end{multline}
  holds.
  On the homological side, $\delta_j \langle  \dlc^r(x),
  Q_{j+p+2(pi-s)(p-1)}\beta  \stp_i(y) \rangle$ is nonzero only if
  $j$ is odd, $r = j+p+2(pi-s)(p-1)$ (implying $r$ even) and either
  \begin{align*}
    2(p-1) \mid (p-1)(\lvert y \rvert -2(p-1)i-1) + j+p+2(pi-s)(p-1) \\
    \intertext{or}
    2(p-1) \mid (p-1)(\lvert y \rvert -2(p-1)i-1) + j+p+2(pi-s)(p-1) + 1.
  \end{align*}
  Note that the second divisibility cannot occur if $j$ is odd, thereby reducing
  these conditions on $j$ simply to being bad. The equality
  \begin{align*}
    \bigl\langle  \dlc^r(x),  \delta_j \sum_{i=0}^{\bigl\lfloor \frac{s-1}{p} \bigr\rfloor} 
    & b_{j,s,\lvert y \rvert}(i)   Q_{j+p+2(pi-s)(p-1)}
    \beta  \stp_i(y) \Biggr\rangle \\
    & =  \begin{cases}
      \delta_j b_{j,s,\lvert y \rvert}(q) \gamma(\lvert x \rvert ,r) \bigl\langle x,\beta  \stp_q(y)
      \bigr\rangle &  q \in \N, j\ \mathrm{bad}\\
      0& \mathrm{else}
    \end{cases}
  \end{align*}
  follows. Since $\lvert \stp^q\beta(x)\rvert =\lvert y \rvert $ and $\delta^r = \delta_j$ imply 
  $b^{r,s,\lvert x\rvert}(q) = b_{j,s,\lvert y \rvert}(q)$, 
  \begin{multline}\label{eq:dl2}
    \bigl\langle \delta^r\sum_{i=0}^{\bigl\lfloor \frac{s-1}{p} \bigr\rfloor}
    b^{r,s,\lvert x\rvert}(i)d^{r,s,\lvert x\rvert}(i)  \dlc^{r-p+2(s-pi)(p-1)} \stp^i\beta (x),
    Q_j(y) \bigr\rangle \\
     =  \bigl\langle  \dlc^r(x), \delta_j \sum_{i=0}^{\bigl\lfloor\frac{s-1}{p} \bigr\rfloor} 
    b_{j,s,\lvert y \rvert}(i)  Q_{j+p+2(pi-s)(p-1)} \beta  \stp_i(y) \bigr\rangle
  \end{multline}
  results.
  Adding equations~(\ref{eq:dl1}) and~(\ref{eq:dl2}) yields the desired result for 
  $r>0$, since the sum of the right sides is just $ \stp_s Q_j(y)$ by Proposition~\ref{hnis}. 
  The extra term in cohomology for $r=0$ does
  not contribute, since
  \[ \Bigl\langle\sum_{n \in N} {\tfrac{1}{e_n}}
  \bigl(\star_{i=1}^p \stp^{n_i}(x)\bigr), Q_j(y)\Bigr\rangle \] 
  is always zero by Proposition~\ref{duality}~\ref{item:pq}. 
  Consider now parings with Browder operations. The left
  hand side is 
  \begin{align*}
    \bigl\langle  \stp^s  \dlc^r (x) ,  L_{n-1}(y,z) \bigr\rangle  & =  
    \bigl\langle  \dlc^r (x),  \stp_s L_{n-1}(y,z) \bigr\rangle \\
    & = \sum_{i=0}^s  \bigl\langle\dlc^r (x), L_{n-1}\bigl(\stp_i(y),\stp_{s-i}(z)\bigr)\bigr\rangle \\
    & =  0
  \end{align*}
  and thus coincides with the right hand side, 
  because $\bigl\langle  \dlc^r(x),  L_{n-1}(y,z) \bigr\rangle = 0=
  \bigl\langle \star_{i=0}^p x_i ,  L_{n-1}(y,z) \bigr\rangle$.

  Finally, pairings with products have to be treated. 
  Proposition~\ref{duality}~\ref{item:qp} implies that only
  $p$-fold products may contribute.
  Let $N^{\prime\prime} =\{n\in \N^p\co n_1+\dotsm+n_p=s\}$. The left
  hand side equals
  \begin{align*}
    \bigl\langle \stp^s \dlc^r(x), \ast_{i=1}^p y_i\bigr\rangle 
    & =  \bigl\langle  \dlc^r(x), \stp_s(\ast_{i=1}^py_i)\bigr\rangle \\
    & = \sum_{n\in N^{\prime\prime}}\bigl\langle \dlc^r(x), \ast_{i=1}^p 
    \stp_{n_i}(y_i)
    \bigr\rangle
  \end{align*}
  which is zero unless $r=0$. In the case $r=0$, the result is
  \begin{align*}
    \bigl\langle  \stp^s \dlc^0(x), \ast_{i=1}^p y_i\bigr\rangle 
    & =  \sum_{n \in \N^{\prime\prime}}\prod_{i=1}^p
    \bigl\langle x,  \stp_{n_i}(y_i) \bigr\rangle.
  \end{align*}
  The right hand side $\bigl\langle  \dlc^{r+2(s-pi)(p-1)}
  \stp^i(x), \ast_{i=0}^p y_i\bigr\rangle$ is zero unless $r+2(s-pi)(p-1)=0$. In
  that case 
  \[ \frac{s}{p} \geq \Bigl\lfloor \frac{s}{p} \Bigr\rfloor \geq i = \frac{r+2(p-1)s}{2(p-1)p} \] 
  which implies $r=0$ (and then also $i = \tfrac{s}{p}$). It follows that
  \begin{multline*}
    \Bigl\langle \sum_{i=0}^{\bigl\lfloor \frac{s}{p} \bigr\rfloor} a^{r,s,\lvert x\rvert}(i)c^{r,s,\lvert x\rvert}(i)
    \dlc^{r+2(s-pi)(p-1)} \stp^i(x), \ast_{i=1}^p y_i\Bigr\rangle
    = \\ \begin{cases}
      \prod_{i=1}^p \bigl\langle  \stp^{\frac{s}{p}}(x), y_i \bigr\rangle & r=0, p\mid s\\
      0            & \mathrm{else}                                                 
    \end{cases}
  \end{multline*}
  because $a^{0,s,\lvert x\rvert}\bigl(\tfrac{s}{p}\bigr) =c^{0,s,\lvert x\rvert}\bigl(\tfrac{s}{p}\bigr)= 1$.
  By the same reasoning, the pairing
  $\bigl\langle  \dlc^{r-p+2(s-pi)(p-1)}
  \stp^i\beta (x) , \ast_{i=1}^p y_i \bigr\rangle$ can be nonzero only if
  $r-p+2(s-pi)(p-1)=0$. But then $\delta^{r} = 0$, so the second term
  \[ \bigl\langle  \delta^r \sum_{i=0}^{\bigl\lfloor \frac{s-1}p \bigr\rfloor}
  b^{r,s,\lvert x\rvert}(i)d^{r,s,\lvert x\rvert}(i)  \dlc^{r-p+2(s-pi)(p-1)} 
  \stp^i\beta (x), \ast_{i=1}^p y_i \bigr\rangle = 0\]
  vanishes completely.
  Again, this proves the claim in the case $r>0$, as both sides of the claimed equation
  pair to zero. In the case $r=0$ one additionally calculates 
  \begin{align*}
    \bigl\langle \sum_{n \in N} \tfrac{1}{e_n} 
    \bigl(\stp^{n_1}(x)\star \dotsm \star \stp^{n_p}(x)\bigr), \ast_{i=1}^p y_i \bigr\rangle 
    & =  \sum_{n \in N}\tfrac{1}{e_n} \sum_{\tau \in \Sigma_p} \prod_{i=1}^p
    \bigl\langle  {\stp^{n_i}}(x),y_{\tau(i)}\bigr\rangle \\
    & =  \sum_{n \in N}\tfrac{1}{e_n} \sum_{\tau \in \Sigma_p} \prod_{i=1}^p
    \bigl\langle  {\stp^{(\tau\cdot n)_i}}(x),y_i\bigr\rangle \\
    & =  \sum_{\tau \in \Sigma_p} \sum_{n \in \tau\cdot N}
    e^{-1}_{\tau^{-1}\cdot n} \prod_{i=1}^p \bigl\langle \stp^{n_i}(x),y_i\bigr\rangle \\
    & =  \sum_{\tau \in \Sigma_p} \sum_{n \in \tau\cdot N} e^{-1}_n
    \prod_{i=1}^p \bigl\langle  \stp^{n_i}(x),y_i\bigr\rangle
  \end{align*}
  since $e_{\tau^{-1}\cdot n} = e_n$,
  where $(\tau \cdot n)_i = n_{\tau^{-1}(i)}$ for $\tau \in \Sigma_p$.
  Let 
  $N^\prime = \{n\in N^{\prime\prime}\co \exists\, 1\leq i,j \leq p \co n_i\neq n_j\}$.  
  For every $n \in \tau \cdot N$ there exist $i,j$ with $n_i\neq
  n_j$ by the defintion of $N$, thus the sums may be combined to give 
  \begin{equation*}
    \bigl\langle \sum_{n \in N} \tfrac{1}{e_n}\bigl(\star_{i=1}^p 
    \stp^{n_i}(x)\bigr), \ast_{i=1}^p y_i \bigr\rangle
     =  \sum_{n\in N^\prime} \frac{\lvert \bigl\{\tau \in \Sigma_p \co n \in t \cdot N
      \bigr\}\rvert}{e_n} \prod_{i=1}^p \bigl\langle 
    \stp^{n_i}(x),y_i\bigr\rangle
  \end{equation*}
  because $e_n\neq 0$ for all such $n$, as mentioned at the 
  beginning of the proof.
  If $\tau\cdot n = \sigma \cdot m$ for $n,m \in N$ and $\tau,
  \sigma \in \Sigma_p$, then $n = m$, since $n$ and $m$ are ascending. 
  Hence for every $n \in \N^p$ there exists an $m \in N$ such that 
  \begin{equation}\label{eq:set}
    \bigl\{\tau \in \Sigma_p \co n \in t \cdot N \bigr\} = 
  \bigl\{\tau \in \Sigma_p \co n = \tau
  \cdot m \bigr\}
  \end{equation}
  and $m$ is unique if this set is nonempty. 
  In that case, it is a coset of the isotropy group of $n$. 
  By the definition of $N$,  the set~(\ref{eq:set}) is empty
  exactly if $\sum n_i \neq s$, whereas
  its cardinality is just $e_n$ if it is nonempty.
  Hence
  \begin{align*}
    \bigl\langle \sum_{n \in N} \tfrac{1}{e_n}\bigl(\star_{i=1}^p 
    \stp^{n_i}(x)\bigr), \ast_{i=1}^p y_i \bigr\rangle
    & =  \sum_{n\in N^{\prime}}\prod_{i=1}^p \bigl\langle \stp^{n_i}(x),y_i\bigr\rangle \\
    &= \begin{cases}
    \sum\limits_{n\in N^{\prime\prime}} \prod\limits_{i=1}^p
    \bigl\langle  \stp^{n_i}(x),y_i\bigr\rangle - \prod\limits_{i=1}^p
    \bigl\langle  \stp^{\frac sp}(x),y_i\bigr\rangle & 
    p \mid s\\
    \sum\limits_{n\in N^{\prime\prime}} \prod_{i=1}^p
    \bigl\langle  \stp^{n_i}(x),y_i\bigr\rangle
    &    p\nmid s
  \end{cases}
  \end{align*}
  Summing these equations finishes
  the proof.
\end{proof}

It remains to supply the proof of the (non)linearity statement~\ref{qcal}.

\begin{proof}
  The second equality follows from the first and Corollary~\ref{cor:div-power}. 
  The proof 
  of the first equation proceeds as the proof of Proposition~\ref{nish}
  given above. It suffices to prove the equations for homogeneous
  elements, so let $x,y\in H^d(D_{n,j}X)$ and $z\in H_d(D_{n,j}X)$.
  Pairing the equation 
  \[ \dlc^r(x+y) = \dlc^r(x)+\dlc^r(y) \]
  with $Q_s(z)$ yields
  \begin{equation*}
   \bigl\langle \dlc^r(x+y) ,Q_s(z)\bigr\rangle 
	    = \begin{cases}
               \gamma(d,r)\bigl\langle x+y ,z \bigr\rangle &   r = s = t(p-1)-\epsilon,\,  \epsilon\in\{0,1\}  \\
		0                                 & \mathrm{otherwise}
               \end{cases}
  \end{equation*}
  and 
  \begin{align*}
    \bigl\langle \! \dlc^r(x)+\dlc^r(y),Q_s(z)\bigr\rangle 
    & = \bigl\langle \dlc^r(x),Q_s(z)\bigr\rangle +
    \bigl\langle \dlc^r(y),Q_s(z)\bigr\rangle  \\
    &= \begin{cases}
               \gamma(d,r)\bigl\langle x+y ,z \bigr\rangle
		&   r\! = \!s \!= \!t(p-1)-\epsilon,\,  \epsilon\in\{0,1\}  \\
		0                                 & \mathrm{otherwise}
               \end{cases}
  \end{align*}
  by part~\ref{item:qq} of Proposition~\ref{duality}, 
  since $\lvert x\rvert = \lvert y\rvert =d$. 
  In the case $r=0$, the summand $\sum\limits_{k=1}^{p-1} \frac{1}{(p-1)!k!} x^{\star k} \star y^{\star(p-k)}$
  pairs trivially with $Q_s(z)$ by  part~\ref{item:pq} of Proposition~\ref{duality}.
  Pairing each of these summands with the value of a homological
  Browder operation also yields zero by Proposition~\ref{duality} 
  part~\ref{item:ql} and part~\ref{item:pl}.

  It remains to check pairings with products. Let
  $z_1,\dotsc,z_k\in H_\ast(D_{n,\ast}X)$. If $r>0$ or $k\neq p$, 
  \[ \bigl\langle \dlc^r(x+y) ,\ast_{i=1}^k z_i \bigr\rangle
  =\bigl\langle \dlc^r(x)+\dlc^r(y) ,\ast_{i=1}^k z_i\bigr\rangle=0\]
  by part~\ref{item:qp} of Proposition~\ref{duality}, and
  if $r=0$ and $k=p$ 
  \begin{align*}
    \bigl\langle \!\dlc^0(x+y) ,\ast_{i=1}^p z_i \bigr\rangle
    & =\prod_{i=1}^p \langle x+y , z_i \rangle\\
    & =\prod_{i=1}^p \bigl(\langle x , z_i \rangle+ \langle y , z_i \rangle\bigr)\\
    &=\prod_{i=1}^p \langle x , z_i \rangle +\!\prod_{i=1}^p \langle y,z_i\rangle
	    + \!\sum_{j=1}^{p-1} \frac{1}{(p-j)!j!} \bigl\langle x^{\star j}y^{\star(p-j)}\! , \ast_{i=1}^p z_i \bigr\rangle
  \end{align*}
  holds by a calculation. 
  Computing the pairing
  \[\bigl\langle \dlc^0(x)+\dlc^0(y)+\sum_{j=1}^{p-1} \frac{1}{(p-j)!j!} x^{\star j}y^{\star (p-j)}, \ast_{i=1}^p z_i \bigr\rangle \]
  yields the same result, which provides the desired equation.
\end{proof}

\end{document}